\documentclass[11pt,a4paper,oneside,reqno]{amsart}
\usepackage[utf8]{inputenc}

\usepackage{amsmath,amscd}
\usepackage{amssymb}
\usepackage{amsthm}
\usepackage[all,cmtip]{xy}

\usepackage{mathrsfs}
\usepackage{hyperref}

\usepackage{graphicx}
\usepackage{psfrag}
\usepackage{comment}
 
\usepackage{latexsym}
\usepackage{amsmath}
\usepackage{amsthm}
\usepackage{amssymb}
\usepackage{amsfonts}
\usepackage{fancyhdr}
\usepackage{comment}
\usepackage{mathrsfs}                           
\usepackage{color}

\usepackage{epstopdf}

\newtheorem{Theorem}{Theorem}[section]
\newtheorem*{TheoremNoNumber}{Theorem}

\newtheorem{Definition}{Definition}[section]

\newtheorem{Lemma}[Theorem]{Lemma}
\newtheorem{Corollary}[Theorem]{Corollary}

\newtheorem{Proposition}[Theorem]{Proposition}
\newtheorem{Remark}[Theorem]{Remark}

\newcommand{\mR}{\mathbb{R}}                    
\newcommand{\abs}[1]{\lvert #1 \rvert}          
\newcommand{\norm}[1]{\lVert #1 \rVert}         

\newcommand{\Mi}{M^{\mathrm{int}}}
\newcommand{\Wi}{W^{\mathrm{int}}}

\newcommand{\mA}{\mathcal{A}}
\newcommand{\mB}{\mathcal{B}}

\newcommand{\R}{\mathbb{R}}
\newcommand{\Tr}[1]{\mbox{Tr}\left( #1 \right)}

\newcommand{\p}{\partial}
\newcommand{\eps}{\varepsilon}
\newcommand{\Norm}[1]{\left\lVert#1\right\rVert}

\title{The Poisson embedding approach to the Calder\'on problem}

\author{Matti Lassas}
\address{Department of Mathematics and Statistics, University of Helsinki}
\email{matti.lassas@helsinki.fi}
\author{Tony Liimatainen}
\address{Department of Mathematics and Statistics, University of Helsinki}
\email{tony.liimatainen@helsinki.fi}
\author{Mikko Salo}
\address{Department of Mathematics and Statistics, University of Jyvaskyla}
\email{mikko.j.salo@jyu.fi}

\begin{document}

\begin{abstract}
We introduce a new approach to the anisotropic Calder\'on problem, based on a map called \emph{Poisson embedding} that identifies the points of a Riemannian manifold with distributions on its boundary. We give a new uniqueness result for a large class of Calder\'on type inverse problems for quasilinear equations in the real analytic case. The approach also leads to a new proof of the result of~\cite{LassasUhlmann} solving the Calder\'on problem on real analytic Riemannian manifolds. The proof uses the Poisson embedding to determine the harmonic functions in the manifold up to a harmonic morphism. The method also involves various Runge approximation results for linear elliptic equations.
\end{abstract}


\maketitle



\date{\today}


\maketitle



\section{Introduction}\label{sec_introduction}
The anisotropic Calder\'on problem consists in determining a conductivity matrix of a medium, up to a change of coordinates fixing the boundary, from electrical voltage and current measurements on the boundary. In dimensions $n \geq 3$ this problem may be written geometrically as the determination of a Riemannian metric on a compact manifold with boundary from Dirichlet and Neumann data of harmonic functions. More precisely, if $(M,g)$ is a compact oriented Riemannian manifold with smooth boundary, we consider the Dirichlet problem for the Laplace-Beltrami operator $\Delta_g$, 
\[
\Delta_g u = 0 \text{ in $M$,} \qquad u|_{\partial M} = f
\]
and denote $u=u_f$. We define the Dirichlet-to-Neumann map (DN map) 
\[
\Lambda_g: C^{\infty}(\partial M) \to C^{\infty}(\partial M), \ \ \Lambda_g f = \partial_{\nu} u_f|_{\partial M}
\]
where $\partial_{\nu}$ is the normal derivative on $\partial M$. One has the coordinate invariance 
\[
\Lambda_g = \Lambda_{\phi^* g}
\]
for any diffeomorphism $\phi: M \to M$ fixing the boundary. 

It is a long-standing conjecture \cite{LU} that if two Riemannian manifolds $(M_1,g_1)$ and $(M_2,g_2)$ with mutual boundary have the same DN maps, then there is a boundary fixing isometry between the manifolds. In this paper we give a candidate for this isometry:

\smallskip
\noindent\framebox{
    \parbox{\textwidth}{%
    \smallskip
    The points $x_1\in M_1$ and $x_2\in M_2$ are to be identified if and only if for every $f\in C^\infty(\p M)$ the harmonic extensions $u_f^1$ and $u_f^2$ of $f$ to $(M_1, g_1)$ and $(M_2, g_2)$ satisfy
    \begin{equation}\label{boxref}
    u_f^1(x_1)=u_f^2(x_2).
    \end{equation}
    }%
}
\smallskip

For general Riemannian manifolds such an identification does not exist. However, if such an identification exists, we show that it induces a mapping $M_1\to M_2$, and that this mapping is the boundary fixing isometry required for the solution of the Calder\'on problem. The Calder\'on problem thus reduces to showing that the equality of DN maps implies that the above identification exists. In this case we say that the manifolds can be \emph{identified by their harmonic functions}.


We introduce a tool for studying the existence of this identification. This will be an embedding $P$ of a Riemannian manifold into the linear dual of the space of smooth functions on its boundary. If $(M,g)$ is a compact Riemannian manifold with $C^\infty$ boundary and $x\in M$, then the value of $P$ at $x$ is a linear functional given by the formula
\[
P(x)f=u_f(x),
\]
where $u_f \in C^{\infty}(M)$ solves the Dirichlet problem
 \begin{align*}
 \Delta_gu_f&=0 \text{ in } M, \\
 u_f&=f \text{ on } \p M.
 \end{align*}
 We call the mapping $P$ the \emph{Poisson embedding}, and it will be the main object of study of this paper. If $P_1(M_1) \subset P_2(M_2)$, then
\begin{equation}\label{tobedefed}
P_2^{-1}\circ P_1
\end{equation}
is a well-defined map $M_1\to M_2$, and we prove that it is a boundary fixing isometry if $n\geq 3$. If $n=2$, it is a conformal mapping. The condition $P_1(M_1) \subset P_2(M_2)$ is equivalent to existence of the identification~\eqref{boxref}.

There are different points of view to the Poisson embedding:
\begin{enumerate}
\item[1.] 
The Poisson formula for solutions of the Dirichlet problem gives that 
\[
P(x) f = \int_{\p M} \partial_{\nu_y} G(x,y) f(y) \,dS(y)
\]
where $G(x,y)$ is the Green function for $\Delta_g$ in $M$ and $\partial_{\nu_y} G(x,\,\cdot\,)$ is the Poisson kernel. Thus $P$ identifies the point $x$ in $M$ with the Poisson kernel $\partial_{\nu_y} G(x,\,\cdot\,)$ on $\p M$ (hence the name Poisson embedding).
\item[2.] 
One also has the formula 
\[
P(x) f = \int_{\p M} f \,d\omega^x
\]
where $\omega^x$ is the \emph{harmonic measure} for $\Delta_g$ at $x$. Thus $P$ identifies points in $M$ with measures on $\p M$; points of $\p M$ are identified with the corresponding Dirac measures, and points in $M^{\mathrm{int}}$ are identified with $C^{\infty}$ functions since $d\omega^x = \partial_{\nu_y} G(x,\,\cdot\,) \,dS$ in this case.
\item[3.]
For $x \in M^{\mathrm{int}}$ one has $G(x,\,\cdot\,) = 0$ on $\p M$, and thus the knowledge of $\partial_{\nu_y} G(x,\,\cdot\,)|_{\p M}$ determines the Green function $G(x,\,\cdot\,)$ in $M$ by elliptic unique continuation. Thus, instead of identifying points of $M$ with the corresponding Green functions in $M$ as in \cite{LassasUhlmann}, we use the normal derivatives of the Green functions on $\p M$. This change of point of view allows one to work on the boundary, which is natural since the measurements are given on $\p M$.
\end{enumerate}

The problem of finding the isometry, or conformal mapping if $n=2$, from the knowledge of the DN map is known as the geometric Calder\'on problem. This problem has been solved in \cite{LassasUhlmann} in the following cases:
\begin{TheoremNoNumber}
Let $(M,g)$ be a compact connected $C^{\infty}$ Riemannian manifold with $C^{\infty}$ boundary.
\begin{itemize}
\item[(a)] 
If $n=2$, the DN map $\Lambda_g$ determines the conformal class of $(M,g)$.
\item[(b)] 
If $n \geq 3$ and if $M$, $\partial M$ and $g$ are real-analytic, then the DN map $\Lambda_g$ determines $(M,g)$.
\end{itemize}
\end{TheoremNoNumber}
As a first application of our techniques we give new proofs of these results. The proofs consist of three steps.

\begin{enumerate}
\item[(1)] The first step is to determine the harmonic functions near the boundary, using a standard boundary determination result~\cite{LU} and real-analyticity. \item[(2)] The second step uses a unique continuation argument for harmonic functions where the manifold and harmonic functions are continued simultaneously. The use of harmonic coordinates (see e.g.~\cite{DeturckKazdan}) and the Runge approximation property are key ingredients in this step. 
\item[(3)] Finally, we show that we can read the metric (conformal metric if $n=2$) from the knowledge of harmonic functions.
\end{enumerate}

The works \cite{LU, LassasUhlmann, LTU} study the Calder\'on problem on real-analytic manifolds, and an analogous result for Einstein manifolds (which are real-analytic in the interior) is proved in \cite{GuSaBar}. See~\cite{LLS} for a recent result for a conformal Calder\'on problem. It remains a major open problem to remove the real-analyticity condition in dimensions $\geq 3$; for recent progress in the case of conformally transversally anisotropic manifolds see \cite{DKSaU, DKLS} and also \cite{DKLLS17, GST18} for the linearized problem. Other interesting approaches for related problems may be found in \cite{Belishev2009} and \cite{Cekic}, and counterexamples for disjoint data are given in \cite{DKN}.

The inverse problems may not be uniquely solvable even when the metric is a priori known to be real analytic.
Indeed,   the above theorem proven in  \cite{LassasUhlmann} does not hold for non-compact manifolds in the
two-dimensional case. It is shown in \cite{LTU} that there are a compact and a non-compact
 complete, two-dimensional manifold for which the
boundary measurements are the same. This counterexample was obtained using a blow-up map. Analogous non-uniqueness results 
have been studied in the invisibility cloaking, where an arbitrary object is hidden from measurements by coating it with
a material that corresponds to a degenerate Riemannian metric 
\cite{GLU1,GKLU1,GKLU2,GKLU3,DLU}.

Our main tool for studying the Poisson embedding and constructing the metric from harmonic functions is the Runge approximation property. This property allows one to approximate local solutions to an elliptic equation by global solutions. In particular this implies that harmonic functions separate points and have prescribed Taylor expansions modulo natural constraints. There are many results of this type in the literature, see e.g.\ \cite{Lax, Malgrange, Browder, RulandSalo}. We require specific approximation results for uniformly elliptic operators, and for completeness they will be given in Appendix ~\ref{runge_apprx_sec} together with proofs. 

\subsection{An inverse problem for quasilinear equations}

As a new result we prove a determination result for Calder\'on type inverse problems for quasilinear equations on manifolds $(M,g)$ with boundary $\p M$, including models that are both anisotropic and nonlinear. We consider the equation
\begin{equation}\label{qlineq}
Q(u)=f \mbox{ in } M, \quad u=0 \mbox{ on } \p M, 
\end{equation}
where $f\in C_c^\infty(W)$, $W\subset M$ is a fixed open set, and $Q$ is a uniformly elliptic quasilinear operator of the form
\begin{equation}\label{Qform}
Qu(x)=\mathcal{A}^{ab}(x,u(x),du(x))\nabla_a\nabla_bu(x)+\mathcal{B}(x,u(x),du(x)).
\end{equation}
Here and below we use Einstein summation rule, which means that repeated indices are always summed over $1,\ldots,n$.
The source-to-solution mapping $S:C_c^\infty(W)\to C^\infty(W)$ for this problem is defined as
\[
 S(f)=u_f|_W
\]
where $u_f$ solves~\eqref{qlineq}. We assume that $(M,g)$ and the matrix valued function $\mathcal{A}$ and the function $\mathcal{B}$ are real analytic. In this case we show that the source-to-solution mapping, even for small data, determines the manifold and the coefficients $\mathcal{A}$ and $\mathcal{B}$ up to a diffeomorphism and a built in ``gauge symmetry'' of the problem:

\begin{Theorem} \label{thm_intro_quasilinear}
Let $(M_1,g_1)$ and $(M_2,g_2)$ be compact connected real analytic Riemannian manifolds with mutual boundary and assume that $Q_j$, $j=1,2$, are quasilinear uniformly elliptic operators of the form~\eqref{Qform} satisfying \eqref{elliptic_a}--\eqref{l_assumption}. Assume that the coefficients $\mathcal{A}_j$, $\mathcal{B}_j$ are real analytic in all their arguments. 
 
Let $W_j\subset M_j^{\mathrm{int}}$, $j=1,2$, be open sets so that there is a diffeomorphism $\phi:W_1\to W_2$, and assume that the local source-to-solution maps $S_j$ for $Q_j$ on $W_j$ agree,
\[
\phi^*S_2f=S_1\phi^*f,
\]
for small $f\in C^\infty_c(W_2)$. 

Then there is a real analytic diffeomorphism $J: \Mi_1 \to \Mi_2$ such that
\[
\mathcal{A}_1=J^*\mathcal{A}_2:=\mathcal{A}
\]
and 
\begin{equation}\label{gaugesym}
\mathcal{B}_1-J^*\mathcal{B}_2=\mathcal{A}^{ab}(\Gamma(g)_{ab}^k-\Gamma(J^*g)_{ab}^k)\sigma_k.
\end{equation}
The mapping $J$ satisfies
\[
J|_{W_1}=\phi: W_1\to W_2.
\]
\end{Theorem}
In Theorem~\ref{thm_intro_quasilinear}, the maps $\mathcal{A}$, $\mathcal{B}_1$ and $J^*\mathcal{B}_2$ have $(x,c,\sigma)\in M_1\times \R \times T^*_xM_1$ as their argument, and $\Gamma(g)_{ab}^k$ and $\Gamma(J^*g)_{ab}^k$ refers to the Christoffel symbols of the metrics $g$ and $J^*g$ respectively. 

The reason why $\mathcal{B}$ can not be determined independently of $\mathcal{A}$ and $\Gamma$, as presented in~\eqref{gaugesym}, is due to the fact that the covariant Hessian in the definition of $Q$ contains first order terms.

The proof proceeds by linearizing the problem and using a slightly modified Poisson embedding for the linearized equation. Using the linearization we can first use the Poisson embedding approach to construct the manifold, but not yet the coefficients $\mathcal{A}$ and $\mathcal{B}$. The source-to-solution mapping determines the coefficients in the measurement set $W$. Since the manifold is now known, the proof is completed by determining the coefficients on the whole manifold by analytic continuation from the set $W$. The linearization method goes back to \cite{I1} and has been used in various inverse problems for nonlinear equations, including anisotropic problems \cite{SuU, HS}. We refer to \cite{Sun_survey, SaloZhong} for further references. 

\subsection{Further aspects of Poisson embedding}

The governing principle of this paper is that instead of trying to find the metric in the anisotropic Calder\'on problem directly, one can focus on finding the harmonic functions. This principle is implemented by Poisson embedding and by the fact, which we prove, that the metric can be determined from the knowledge of harmonic functions. We will now discuss in more detail some aspects the Poisson embedding approach.

If $J$ is an isometry $(M_1,g_1)\to (M_2,g_2)$, then $J$ can be locally represented in various coordinate charts. Useful coordinate charts include boundary normal coordinates and harmonic coordinates~\cite{DeturckKazdan, Taylor_isometry}. In the study of the Calder\'on problem, boundary normal coordinates have been used to locally identify real analytic manifolds near boundary points by showing that in boundary normal coordinates the metrics $g_1$ and $g_2$ agree. See e.g.~\cite{LU, LassasUhlmann, LTU, GuSaBar, LLS}. However, local representations do not yield a global candidate for the isometry $J$ required for solving the Calder\'on problem. In contrast, the Poisson embeddings $P_j$ are globally defined objects, and they yield a candidate $P_2^{-1} \circ P_1$ for the isometry (in fact we prove that if $P_2^{-1} \circ P_1$ is well defined, then it gives the required isometry). The representation formula $P_2^{-1}\circ P_1$ also gives uniqueness of the boundary fixing isometry directly if it exists. 

In~\cite{LTU} the authors introduce an embedding of real analytic Riemannian manifolds by using Green functions to study the Calder\'on problem. Their approach involves an analytic continuation argument based on the implicit function theorem applied to the embedding. In contrast, the analogous step for Poisson embedding can be done simply by using harmonic coordinates. This feature also emphasizes the role of choosing suitable coordinates in the study of the anisotropic Calder\'on problem. Moreover, the recovery of the metric using Poisson embedding is an elementary linear algebra argument that yields a representation of the metric in terms of harmonic functions (see Proposition~\ref{morphism_to_homothety}). In~\cite{LTU} an asymptotic expansion of Green functions near the diagonal is used.

The basic principle of Poisson embedding is to control the points on a manifold by the values of solutions to Dirichlet problems on the manifold. This principle generalizes to nonlinear equations, to linear systems or to nonlocal operators where Green functions might not be easily accessible. 
The Poisson embedding also generalizes directly to less regular, say piecewise smooth or $C^k$ regular, settings. 


Finally, we mention that ideas related to the Poisson embedding have been used in other fields as well. In~\cite{GW} an embedding by finitely many harmonic functions is used to embed an open Riemannian manifold into a higher dimensional Euclidean space. Their embedding is similar to the Poisson embedding, in the sense that the Poisson embedding parametrizes the manifold by the data of all, instead of a finite number, of harmonic functions on the manifold. Another related result is that when the mapping $P_2^{-1}\circ P_1:M_1\to M_2$ exists, it is necessarily a \emph{harmonic morphism}. The study of harmonic morphisms, which are mappings that preserve harmonic functions, has applications in the study of minimal surfaces and in mathematical physics~\cite{Wood_book}. In Section \ref{sec_recovery_of_metric} we give a new proof of the characterization of harmonic morphisms as homotheties~\cite[Corollary 3.5.2]{Wood_book} based on harmonic coordinates.

\subsection{Outline of the paper}
In Section~\ref{sec_poisson_embed} we introduce the Poisson embedding and its basic properties, in a way that does not directly involve the Calder\'on problem. In Section~\ref{sec_det_of_harmonic_ftions} we determine the harmonic functions from the knowledge of the DN map on real analytic manifolds by using the Poisson embedding. From the knowledge of harmonic functions, we then determine the metric in Section~\ref{sec_recovery_of_metric}, which gives a new proof of the main result of~\cite{LU} in dimension $n\geq 3$. Section \ref{sec_calderon_2d} gives a new proof of the two-dimensional result of~\cite{LU}. In Section~\ref{sec:q_lin} we use the Poisson embedding approach and linearization to prove Theorem \ref{thm_intro_quasilinear}, yielding uniqueness in the inverse problem for quasilinear equations. Appendix~\ref{runge_apprx_sec} is independent of the rest of the paper and contains Runge approximation results. In Appendix~\ref{sec_appendix} we include the proofs of some basic auxiliary results.

\subsection*{Acknowledgements}
M.L., T.L.\ and M.S.\ were supported by the Academy of Finland (Centre of Excellence in Inverse Modelling and Imaging, grant numbers 284715 and 309963). M.S.\ was also partly supported by the European Research Council under FP7/2007-2013 (ERC StG 307023) and Horizon 2020 (ERC CoG 770924).

\section{Poisson embedding}\label{sec_poisson_embed}
We begin by introducing the Poisson embedding of a Riemannian manifold $(M,g)$ with boundary. 
Throughout this section, we assume that $M$ is connected and $M$ and $g$ are $C^\infty$. In particular, we do not require real analyticity in the definition of the Poisson embedding. We will solve Dirichlet problems with boundary values supported on a nonempty open set $\Gamma$ of the boundary $\p M$. The domain of the Poisson embedding will be
\[
 M^\Gamma:= \Mi \cup \Gamma\subset M.
\]

\begin{Definition}[Poisson embedding]\label{Poisson_embed_Def}
 Let $(M,g)$ be a compact Riemannian manifold with boundary, and let $\Gamma$ be a nonempty open subset of $\p M$. The \emph{Poisson embedding} of the manifold $M$ is defined to be the mapping 
 \[
 P:M^\Gamma\to \mathcal{D}'(\Gamma), \quad P(x)f=u_f(x)
 \]
 where $u_f(x)$ solves the Dirichlet problem
 \begin{align*}
 \Delta_g u_f&=0 \text{ in } M, \\
 u_f &=f \text{ on } \p M
 \end{align*}
 with $f\in C^\infty_c(\Gamma)$.
\end{Definition}

In the definition $\mathcal{D}'(\Gamma)$ is the space of distributions on $\Gamma$, i.e.\ $\mathcal{D}'(\Gamma)$ is the dual of $C^{\infty}_c(\Gamma)$. We call $P$ the Poisson embedding due to the connection with the representation formula for the solution $u_f(x)$ in terms of the \emph{Poisson kernel}
\[
\p_{\nu_y} G(x,y), 
\]
where $x\in M \mbox{ and } y\in \p M$,
as
\[
 u_f(x)=\int_{\p M} \p_{\nu_y}G(x,y)f(y) \,dS_{g_{\p M}}.
\]
Here $G(x,y)$ is the Dirichlet Green's function of $(M,g)$ and $dS_{g_{\p M}}$ is the induced metric on the boundary. Thus $P(x)$ can be identified with the distribution $\p_{\nu_y} G(x,\,\cdot\,)$ on $\Gamma$. In fact, if $x \in \Mi$ then $\p_{\nu_y} G(x,\,\cdot\,)$ is in $C^{\infty}(\Gamma)$, and if $x \in \Gamma$ then $\p_{\nu_y} G(x,\,\cdot\,) = \delta_x(\,\cdot\,)$ so $P(x)$ is always a measure on $\Gamma$. 

We have the following basic properties of $P$. 

 \begin{Proposition}\label{P_basics}
 Let $(M,g)$ be a compact manifold with boundary. For any $x \in M^\Gamma$, one has $P(x) \in H^{-s}(\Gamma)$ whenever $s+1/2>n/2$. 
 The mapping $P$ is continuous $M^\Gamma \to H^{-s-1}(\Gamma)$ and $k$ times Fr\'echet differentiable considered as a mapping $M^\Gamma \to H^{-s-1-k}(\Gamma)$. In particular, 
 $P:M^\Gamma \to \mathcal{D}'(\Gamma)$ is $C^\infty$ smooth in the Fr\'echet sense. 
 
 The Fr\'echet derivative of $P$ at $x$ is a linear mapping given by
 \begin{equation}\label{formula_for_the_differential}
 DP_x:T_xM^\Gamma\to \mathcal{D}'(\Gamma), \quad (DP_xV)f=du_f(x)\cdot V,
 \end{equation}
 where $u_f$ solves $\Delta_gu_f=0$ in $M$ and $u_f|_{\p M} = f \in C^{\infty}_c(\Gamma)$.
\end{Proposition}
In the proposition $\cdot$ refers to the canonical pairing of vectors and covectors on $M^\Gamma$, that is, $du_f(x)\cdot V\equiv [du_f(x)](V)=\p_au_f(x)V^a$.

We have placed the proof of the proposition, which just applies standard estimates for solutions of elliptic equations, in Appendix \ref{sec_appendix}. However, let us formally calculate where the formula for the derivative of $P$ comes from. Let $x\in M^\Gamma$ and $V\in T_xM^\Gamma$. By definition $V$ is given by a path $\gamma:[0,1]\to M^\Gamma$ such that $\gamma(0)=x$ and $\frac{d}{dt}\big|_{t=0}\gamma(t)=V$. Let $f\in C_c^\infty(\Gamma)$. A formal calculation of $DP_xV\in \mathcal{D}'(\Gamma)$ now gives
\[
[DP_xV]f=\left[\frac{d}{dt}\Big|_{t=0}P\circ \gamma(t)\right]f=\frac{d}{dt}\Big|_{t=0}u_f(\gamma(t))=\p_au_f(x)V^a=du_f(x)\cdot V.
\]

Let us next show that the mapping $P$ is indeed an embedding, which means a $C^{\infty}$ injective mapping with injective Fr\'echet derivative. 
The main tool that we encounter here for the first time is Runge approximation, which allows one to approximate locally defined harmonic functions by global harmonic functions. It is known since \cite{Lax, Malgrange} that approximation results of this type follow by duality from the unique continuation principle. We have devoted Appendix~\ref{runge_apprx_sec} to various Runge approximation results. We will mostly use the following consequence, whose proof may also be found in Appendix~\ref{runge_apprx_sec}.

\begin{Proposition} \label{prop_runge_consequence_first}
Let $(M,g)$ be a compact Riemannian manifold with smooth boundary, and let $\Gamma$ be a nonempty open subset of $\partial M$.
\begin{enumerate}
\item[(a)] 
If $x \in M^{\Gamma}$, $y \in M$ and $x \neq y$, there is $f \in C^{\infty}_c(\Gamma)$ such that 
\[
u_f(x) \neq u_f(y).
\]
\item[(b)] 
If $x \in M^{\Gamma}$ and $v \in T_x^* M$, there is $f \in C^{\infty}_c(\Gamma)$ such that 
\[
du_f(x) = v.
\]
\end{enumerate}
\end{Proposition}

 
\begin{Proposition}[$P$ is an embedding]\label{P_is_embedding}
Let $(M,g)$ be a compact manifold with boundary. The mapping $P:M^\Gamma\to \mathcal{D}'(\Gamma)$ is a $C^\infty$ embedding in the sense that it is injective with injective Fr\'echet derivative on $TM^\Gamma$.
\end{Proposition}
\begin{proof} 
Let $x_1,x_2\in M^\Gamma$ and assume that $P(x_1)=P(x_2)$. That is, for all boundary value functions $f\in C^\infty_c(\Gamma)$, we have
 \[
 u_f(x_1)=P_1(x_1)f=P_2(x_2)f=u_f(x_2).
 \]
We need to show that $x_1=x_2$. We argue by contradiction and suppose that $x_1\neq x_2$. But by Proposition \ref{prop_runge_consequence_first} there is $f_0 \in C^{\infty}_c(\Gamma)$ such that 
\[
 u_{f_0}(x_1)\neq u_{f_0}(x_2).
\]
This is a contradiction and we must have $x_1=x_2$. Thus $P$ is injective.
 
To show that the differential of $P$ is injective, let $x\in M^\Gamma$ and $V\in T_xM^\Gamma$, and assume that $DP_xV=0$. By the formula~\eqref{formula_for_the_differential} for the differential $DP_x$, we have
 \begin{equation}\label{Dinj_eq}
 [DP_xV]f=du_f(x)\cdot V=0,
 \end{equation}
 for all boundary value functions $f\in C_c^\infty(\Gamma)$. To conclude that $V=0$, we use Proposition \ref{prop_runge_consequence_first} again and choose $f \in C^{\infty}_c(\Gamma)$ such that $\nabla u_f(x) = V$. Here $\nabla u$ is the Riemannian gradient. 
 Thus the condition~\eqref{Dinj_eq} yields 
 \[
 0=du_{f}(x)\cdot V = \abs{V}_{g(x)}^2
\]
 showing that $V=0$. Thus the differential of $P$ is injective on $TM^\Gamma$.
\end{proof}

\subsection{Composition of Poisson embeddings}
Let $(M_1,g_1)$ and $(M_2, g_2)$ be compact manifolds with mutual boundary $\p M$. One of the aims of this paper is to give a candidate for the isometry that one hopes to construct in the anisotropic Calder\'on problem. 
This candidate is 
\[
J:=P_2^{-1}\circ P_1.
\]
We will see that, whenever this mapping is well defined from $M_1$ to $M_2$, it is exactly the mapping that one seeks in the Calder\'on problem. In dimensions $n\geq 3$, it is an isometry. In dimension $2$, it is a conformal mapping. It also fixes the boundary if we consider full data problem, or the part $\Gamma\subset \p M$, if we consider partial data problem with measurements made on $\Gamma$.

We now begin to study the composition $P_2^{-1}\circ P_1$. We include considerations related to partial data problem on the part $\Gamma$ of the boundary. For this purpose we denote throughout this section
\[
 M_1^\Gamma= \Mi_1 \cup \Gamma \text{ and }M_2^\Gamma= \Mi_2 \cup \Gamma.
\]
Since we already know that the Poisson embeddings $P_j$, $j=1,2$, are injective, we know that the mapping is well defined and bijective if the image sets of $P_j:(M_j^\Gamma,g_j)\to \mathcal{D}'(\Gamma)$ agree:
\[
P_1(M_1^\Gamma)=P_2(M_2^\Gamma).
\]
Thus solving the Calder\'on problem would reduce to verifying this condition from the knowledge that the DN maps of $(M_1,g_1)$ and $(M_2,g_2)$ agree. 

The next lemma considers the smoothness properties of the mapping $J$ assuming it is defined on some open set of $M_1$. 
\begin{Lemma}\label{J_is_local_dif}
 Let $(M_{j},g_j)$, $j=1,2$, be compact manifolds with mutual boundary $\p M$. Assume that 
 \[
 P_1(B) \subset P_2(M_2^\Gamma)
 \]
for some open set $B\subset M_1^\Gamma$. Then $J(B) \subset M_2^{\Gamma}$, and $J = P_2^{-1} \circ P_1$ is a $C^\infty$ diffeomorphism $B\to J(B)$.
\end{Lemma}
\begin{proof}
We first note that, writing $u_f^j(x) = P_j(x) f$, for any $f \in C_c^{\infty}(\Gamma)$ one has 
\begin{equation} \label{p1_p2_relation}
u_f^1(x) = u_f^2(J(x)), \qquad x \in B.
\end{equation}
This follows from the computation $P_1(x) =P_2(P_2^{-1}\circ P_1(x))=P_2\circ J(x)$. Now, to show that $J(B) \subset M_2^{\Gamma}$, we argue by contradiction and assume that there is some $x \in B \subset M_1^{\Gamma}$ with $J(x) \in \partial M \setminus \Gamma$. But then by \eqref{p1_p2_relation} 
\[
u_f^1(x) = u_f^2(J(x)) = 0, \qquad f \in C^{\infty}_c(\Gamma),
\]
which is impossible by Proposition \ref{prop_runge_consequence_first}. Thus $J(B) \subset M_2^{\Gamma}$.

We next prove that $J: B \to M_2$ is continuous (this perhaps surprisingly uses compactness of $M_2$). Let $x\in B$. If $J$ would not be continuous at $x$, there would be $\eps>0$ and a sequence $(x_l) \subset M_1$ with $x_l \to x$ such that $J(x_{l})\notin B(J(x),\eps)$. By the compactness of $M_2$, passing to a subsequence (still denoted by $(x_l)$), we may assume that $J(x_{l})$ converges to $y\in M_2$. We have $d(y,J(x))\geq \eps$.
 
Now, using Proposition \ref{prop_runge_consequence_first} in $M_2$ and the fact that $J(x) \in M_2^{\Gamma}$, we can find $f \in C_c^{\infty}(\Gamma)$ such that the harmonic function $u_f^2 \in C^{\infty}(M_2)$ satisfies $u_f^2(J(x))  \neq u_f^2(y)$. The formula \eqref{p1_p2_relation} shows that 
\[
 u_f^1(x_{l})-u_f^1(x)=u_f^2(J(x_{l}))-u_f^2(J(x)).
\]
 Since harmonic functions are continuous, taking the limit $l \to \infty$ yields
 \[
 0 = u_f^2(y) - u_f^2(J(x)),
\]
 which is a contradiction. Thus $J$ is continuous. 

We will next show that $J: B \to M_2$ is $C^{\infty}$. This follows an idea from \cite{Taylor_isometry} related to smoothness of Riemannian isometries. Fix $x \in B$, and choose 
harmonic coordinates $U=(u_{f_1}^2,\ldots, u_{f_n}^2)$ with $f_j\in C_c^\infty(\Gamma)$, $j=1,\ldots,n$, near $J(x)$. This can be done by Proposition \ref{prop_runge_consequence_first} upon choosing $\{ d u_{f_1}^2(J(x)), \ldots, du_{f_n}^2(J(x)) \}$ linearly independent. Write $V = (u_{f_1}^1,\ldots, u_{f_n}^2)$. By the formula \eqref{p1_p2_relation} we have $V = U \circ J$ in $B$.  Now $U$ is bijective in some neighborhood $\Omega \subset M_2$ of $J(x)$, and since $J$ is continuous there is a neighborhood $W$ of $x$ with $J(W) \subset \Omega$. Thus we have 
 \begin{equation}\label{representation_for_J}
J = U^{-1}\circ V \text{ in $W$.}
 \end{equation}
Since the harmonic functions $u_{f_k}^j$, $j=1,2$, are smooth, the smoothness of $J$ near $x$ follows. 
 

It remains to show that the differential of $J$ is invertible on $B$. Since $J: B \to M_2$ is injective, the claim will then follow from the inverse function theorem. Let $f\in C_c^\infty(\Gamma)$, $x\in B$ and $X\in T_xM_1$. By \eqref{p1_p2_relation}, we have $u_f^1=u_f^2\circ J$. Together with \eqref{formula_for_the_differential} this gives: 
 \begin{align*}
 [DP_1(x)X]f&
 =X\cdot du_f^1(x)=X\cdot d(u_f^2\circ J)(x)=X\cdot (du_f^2|_{J(x)} (DJ)^T)   \\
 &= du_f^2|_{J(x)}\cdot (DJ(x))X.
 \end{align*}
 The left hand side is equal to $d u_f^1\cdot X$. Thus for any $f \in C^{\infty}_c(\Gamma)$ we have the equation
 \begin{equation}\label{equation_of_injectivity}
 d u_f^1|_x \cdot X=d u_f^2|_{J(x)} \cdot (DJ(x))X.
 \end{equation}
 This equation will be used again later on and we name it ``the equation of injectivity''. (This equation can be interpreted as the infinite dimensional counterpart of the chain rule for the composition $P_2\circ (P_2^{-1}\circ P_1)$.)
 
 Assume that $(DJ(x))X=0$ and choose by Proposition \ref{prop_runge_consequence_first} a harmonic function $u_f^1$ so that the Riemannian gradient satisfies $\nabla u_f^1(x)= X$. Then 
 \[
 |X|^2=d u_f^2\cdot (DJ(x)) X=0.
 \]
 Thus $X=0$. It follows that $DJ(x)$ is injective at $x$ and since the manifolds are of the same dimension, it is invertible. This proves the claim.
\end{proof}

\section{Determination of harmonic functions}\label{sec_det_of_harmonic_ftions}

We have so far acquired the basic properties of the Poisson embedding. We now move on to give a new proof of the fact~\cite{LassasUhlmann} that for real analytic Riemannian manifolds with $\dim(M) \geq 3$, the knowledge of the DN-map determines the Riemannian manifold up to isometry. Throughout this section we will assume that the manifolds $M_j$ and metrics $g_j$, $j=1,2$, are real analytic, and $n = \dim(M_j) \geq 3$. We continue to denote $M_j^\Gamma= \Mi_j \cup \Gamma$.

We first show that near any boundary point there exist coordinates in which the coordinate representations of harmonic functions, corresponding to a common boundary value $f$, agree. This follows from boundary determination~\cite{LU} and unique continuation. 

\begin{Lemma}[Determination near the boundary] \label{lemma_determination_near_the_boundary}
 Let $(M_1,g_1)$ and $(M_2,g_2)$ be compact real analytic manifolds with mutual boundary whose DN maps agree on an open set $\Gamma \subset \p M$. Assume also that $\Gamma$ is real analytic. Then, for any $p\in \Gamma$ there are boundary normal coordinates $\psi_j$, $j=1,2$, defined on neighborhoods $U_j\subset M_j^\Gamma$ of $p$, such that $\psi_1$ and $\psi_2$ agree on $\Gamma$ and such that for any boundary function $f\in C_c^\infty(\Gamma)$ we have
 \[
 u_f^1\circ \psi_1^{-1}(x)=u_f^2\circ \psi_2^{-1}(x), \quad x\in \psi_1(U_1)\cap \psi_2(U_2)\subset \{x^n\geq 0\}.
 \]
 Here the functions $u_f^j\circ \psi_j^{-1}$ are the coordinate representations of the harmonic functions $u_f^j$ on $(M_j,g_j)$ with boundary value $f$.
\end{Lemma}
\begin{proof}
 Let $p\in \Gamma\subset \p M$ and let $\psi_j$, $j=1,2$, be boundary normal coordinates near $p$ on manifolds $(M_j,g_j)$, respectively, so that $\psi_1|_\Gamma=\psi_2|_\Gamma$. Then by the boundary determination result in~\cite{LU}, we have that in these coordinates, the jets of the Riemannian metrics $g_j$ agree. Since $g_j$ and $\Gamma$ are real analytic, it follows that $\psi_j$ are real analytic coordinate charts, and thus the coordinate representations $\psi_j^{-1*}g_j$ of $g_j$ agree near $x=\psi_1(p)=\psi_2(p)\in \{ x_n = 0 \}$. 
 
 Write $g=\psi_1^{-1*}g_1=\psi_2^{-1*}g_2$. If $f$ is any $C^\infty_c(\Gamma)$ boundary function, we have that $\tilde{u}_f^1=u_f^1\circ \psi_1^{-1}(x)$ and $\tilde{u}_f^2=u_f^2\circ \psi_2^{-1}(x)$ satisfy the same elliptic equation
 \[
 \Delta_g\tilde{u}_f^i=0 \text{ in $\psi_1(U_1)\cap \psi_2(U_2)$}
 \]
 with the same Cauchy data (since the DN-maps agree)
 \[
 \tilde{u}_f^1=\tilde{u}_f^2, \ \ \p_{x_n} \tilde{u}_f^1=\p_{x_n} \tilde{u}_f^2 \mbox { on } \psi_1(\Gamma)\cap \psi_2(\Gamma).
 \]
 Thus by elliptic unique continuation \cite[Theorem 3.3.1]{Isakov_book}, we have 
 \[
 \tilde{u}_f^1=\tilde{u}_f^2 \mbox{ on } \psi_1(U_1)\cap \psi_2(U_2)
 \]
 as claimed.
 \end{proof}

\begin{Lemma} \label{local_determination_on_B}
Under the same assumptions as in Lemma~\ref{lemma_determination_near_the_boundary}, there exists an open set $B \subset M_1^\Gamma$, which contains all points of $\Gamma$, and a $C^{\infty}$ diffeomorphism $F: B \to F(B) \subset M_2^\Gamma$ such that 
 \begin{gather*}
  P_1(B) \subset P_2(M_2^\Gamma), \qquad  P_1 = P_2 \circ F \text{ on $B$}.
\end{gather*}
\end{Lemma}
\begin{proof}
 By Lemma~\ref{lemma_determination_near_the_boundary}, in its notation, we have that for any $p \in \Gamma$ there exist $U_j\subset M_j^\Gamma$ such that for all $f\in C_c^\infty(\Gamma)$, we have
 \[
 u_f^1\circ \psi_1^{-1}=u_f^2\circ \psi_2^{-1} \text{ on } \psi_1(U_1)\cap \psi_2(U_2).
 \]
 Thus for $x\in \psi_1^{-1}(\psi_1(U_1) \cap \psi_2(U_2))$, we have
 \[
  P_1(x)f=u_f^1(x)=u_f^2 (\psi_2^{-1}\circ \psi_1(x))=P_2(\psi_2^{-1}\circ \psi_1(x))f.
 \]
Setting $B_p=\psi_1^{-1}(\psi_1(U_1) \cap \psi_2(U_2))$ and 
\[
B = \bigcup_{p \in \Gamma} B_p
\]
gives an open set $B \subset M_1^{\Gamma}$ such that $\Gamma \subset B$ and $P_1(B) \subset P_2(M_2^{\Gamma})$. By Lemma \ref{J_is_local_dif} it is enough to set $F = P_2^{-1} \circ P_1$ in $B$.
\end{proof}

We have now shown that knowledge of the DN map on $\Gamma$ determines harmonic functions near $\Gamma$. We proceed in the real analytic case to determine the harmonic functions globally. From this knowledge we will then determine Riemannian manifolds up to isometry in Section~\ref{sec_recovery_of_metric}. 

In the following result a (global) \emph{harmonic morphism} means a mapping that preserves solutions to the Dirichlet problem. Precisely, a $C^{\infty}$ mapping $H: (M_1^{\Gamma}, g_1) \to (M_2^{\Gamma}, g_2)$ is a harmonic morphism, if for any $f\in C_c^\infty(\Gamma)$ we have
\[
u_f^1=u_f^2\circ H.
\]
Here $u_f^1$ and $u_f^2$ are the solutions to the Dirichlet problems for $\Delta_{g_1}$ and $\Delta_{g_2}$ as usual. 
For more details on harmonic morphisms, we refer to~\cite{Wood_book}. (In~\cite{Wood_book} a harmonic morphism is a mapping that also preserves \emph{local} harmonic functions instead of just global harmonic functions, but our results will show that there is no difference at least when $\dim(M_1)=\dim(M_2)$.)

\begin{Theorem} \label{global_determination_of_harmonic_functions}
Let $(M_1, g_1)$ and $(M_2, g_2)$ be compact real analytic manifolds with mutual boundary.
Let $\Gamma$ be an open subset of $\p M$.
Assume that there is a neighborhood $B\subset M_1^\Gamma$ of a boundary point $p\in \Gamma$ 
and a mapping $F:B\to F(B)\subset M_2^\Gamma$ diffeomorphic onto its image such that
 \[
 P_1=P_2\circ F \mbox{ on } B.
 \]
 Then we have $P_1(M_1^\Gamma)=P_2(M_2^\Gamma)$ and 
 \[
 J:=P_2^{-1}\circ P_1:M_1^\Gamma\to M_2^\Gamma
 \]
 is a diffeomorphic global harmonic morphism extending $F$.
\end{Theorem}
\begin{proof}
 Redefine $B$ to be the largest connected open subset of $M_1^{\Gamma}$ such that $P_1(B)\subset P_2(M_2^\Gamma)$. By our assumption $B$ is nonempty. 
 We will show that $B$ is closed and thus $B=M_1^{\Gamma}$. Then $P_1(M_1^\Gamma) \subset P_2(M_2^\Gamma)$, and from Lemma \ref{J_is_local_dif} it will follow that
 \[
 J=P_2^{-1}\circ P_1: M_1^\Gamma \to M_2^\Gamma
 \]
 is a well defined $C^{\infty}$ diffeomorphism. It is also a harmonic morphism, since for any $f\in C_c^\infty(\Gamma)$ and $x \in M_1^\Gamma$ we would have
 \begin{equation}\label{J_is_harm_morph}
 J(x)=P_2^{-1}(P_1(x)) \Rightarrow P_2(J(x))f=P_1(x)f \iff u_f^2(J(x))=u_f^1(x).
 \end{equation}
 This would prove the claim.
 
 We argue by contradiction and assume that $B$ is not closed. Then there is a sequence $(p_k)$ in $B$ with $p_k \to x_1$ as $k\to \infty$, where $x_1 \in \p B \setminus B\subset M_1^\Gamma$. After passing to a subsequence, there is $x_2 \in M_2$ such that 
 \[
 x_2 = \lim_{k\to\infty}J(p_k).
 \]
 We actually have that $x_2\in M_2^\Gamma$. This is because if $x_2\in \p M\setminus \Gamma$, then~\eqref{p1_p2_relation} applied at $x=p_k$ shows that for all $f\in C_c^\infty(\Gamma)$ we have
 \begin{equation}\label{x2_not_bndr}
 u_f^1(x_1)=\lim u_f^1(p_k)=\lim u_f^2(J(p_k))=u_f^2(x_2)=f(x_2)=0.
 \end{equation}
 This cannot be true by Proposition \ref{prop_runge_consequence_first}, so indeed $x_2$ must be in $M_2^{\Gamma}$.
 
 Let 
 \[
 U=(u_{f_1}^2,u_{f_2}^2,\ldots, u_{f_n}^2)
 \]
 be harmonic coordinates on a neighborhood $\Omega_2\subset M_2^{\Gamma}$ of $x_2$, as in Lemma~\ref{J_is_local_dif}. Here $u_{f_l}^2$, $l=1,\ldots,n$, are global harmonic functions on $M_2$ with boundary values $f_l\in C_c^\infty(\Gamma)$.

The mapping $J=P_2^{-1}\circ P_1: B\to J(B)$ is well defined by Lemma \ref{J_is_local_dif}. Now comes the punch line: we set
 \[
 V=(u_{f_1}^1,u_{f_2}^1,\ldots, u_{f_n}^1).
 \]
We will prove shortly that $V$ is a coordinate system in some neighborhood $\Omega_1 \subset M_1^{\Gamma}$ of $x_1$. The map $U^{-1} \circ V$ will then give us a local identification of neighborhoods $\Omega_1\subset M_1^\Gamma$ of $x_1$ and $\Omega_2\subset M_2^\Gamma$ of $x_2$, such that $\Omega_1$ intersects the complement of the closure of $B$ in $M_1^\Gamma$ (if nonempty). This will allow us to extend $B$ and reach a contradiction.
 
 We now take a small deviation from the main line of the proof, and show that $V$ is also a harmonic coordinate system around $x_1$. To see this, first observe that $V(x_1)=U(x_2)$ since
 \[
  u_{f_l}^1(x_1)=\lim_k u_{f_l}^1(p_k)=\lim_k u_{f_l}^2(J(p_k))=u_{f_l}^2(x_2), \quad l=1,\ldots,n.
 \]
Next we note that we have at $p_k$ the equation of injectivity~\eqref{equation_of_injectivity}
 \begin{equation}\label{A}
 d u_f^1(p_k)\cdot X_k=d u_f^2|_{J(p_k)}\cdot D(P_2^{-1}\circ P_1)(p_k)X_k.
 \end{equation}
 This holds for all $X_k\in T_{p_k}M_1$. Since near $x_1$ on $B$, we have $J=U^{-1}\circ V$, see~\eqref{representation_for_J}, we can substitute this to the equation of injectivity yielding
 \[
 d u_f^1(p_k)\cdot X_k=d u_f^2|_{J(p_k)}\cdot D(U^{-1}\circ V)(p_k)X_k.
 \]
 
 Assume that $X\in T_{x_1}M$ is such that $D(U^{-1}\circ V)(x_1)X=0$. Let us take a sequence $X_k\to X$ in $TM$ as $k\to \infty$. Note that $D(U^{-1}\circ V)$ is a continuous (in fact smooth) matrix field even though we still do not know if it is invertible. Taking the limit as $k\to \infty$ in~\eqref{A} gives
 \[
 d u_f^1(x_1)\cdot X=d u_f^2(x_2)\cdot D(U^{-1}\circ V)(x_1)X.
 \]
 Choosing $f$ so that $\nabla u_f^1(x_1)=X$ (by Proposition \ref{prop_runge_consequence_first}) shows that $X=0$ and consequently that $D(U^{-1}\circ V)(x_1)$ is invertible. Since $U$ is a local diffeomorphism, it follows that $DV(x_1)$ is invertible. Thus $V$ is also a coordinate system near $x_1$ as claimed.
 
Let us continue on the main line of the proof of the proposition. Let $\Omega_1 \subset M_1^{\Gamma}$ be a neighborhood of $x_1$ so that $V$ is a coordinate system in $\Omega_1$, and redefine $\Omega_1$, if necessary, so that $V(\Omega_1)\subset U(\Omega_2)$. On $B$ we have by \eqref{p1_p2_relation}
 \begin{equation}\label{repr_on_B}
 V=U\circ J.
 \end{equation}
 What we will show is that
 \begin{equation}\label{to_be_shown}
 P_1(x)=P_2(U^{-1}\circ V(x)), \quad x\in \Omega_1.
 \end{equation}
 Note that we require~\eqref{to_be_shown} to hold in $\Omega_1$, not only in $B\cap \Omega_1$.
 In particular we will have
 \[
 P_1(B\cup \Omega_1)\subset P_2(M_2^\Gamma).
 \]
 Since $\Omega_1$ in this case extends $B$ to a neighborhood of the point $x_1\in \p B\setminus B$, this will give a contradiction and prove the theorem since $M_1^\Gamma$ is connected. 
 
 So far we have not used the assumption of real analyticity, but we use it now to prove~\eqref{to_be_shown}. To prove it, let $f\in C_c^\infty(\Gamma)$. Now $u_f^j\in C^\omega(\Mi_j)$ and $V$ and $U$ are local $C^{\omega}$ diffeomorphisms in $\Omega_1 \cap \Mi_1$ and $\Omega_2 \cap \Mi_2$, respectively. Thus we have
 \begin{equation}\label{ra}
 u_f^1\circ V^{-1}, u_f^2\circ U^{-1}\in C^\omega(V(\Omega_1 \cap \Mi_1))
 \end{equation}
where $V(\Omega_1 \cap \Mi_1) \subset \R^n$. Since on $B\cap\Omega_1$, we have by~\eqref{J_is_harm_morph} and~\eqref{repr_on_B} that
 \[
 u_f^1=u_f^2\circ J  \mbox{ and } J=U^{-1}\circ V
 \]
 it follows that
 \[
 u_f^1\circ V^{-1}=u_f^2\circ U^{-1} \mbox{ on } V(B\cap\Omega_1).
 \]
 By the real analyticity, stated in~\eqref{ra}, this holds on the whole $V(\Omega_1)$. This means that the coordinate representations of $u_f^1$ and $u_f^2$ in the (harmonic) coordinates $V$ and $U$ agree. 
 
 Given any $f \in C^{\infty}_c(\Gamma)$, we have the chain of equivalences
 \begin{align*}
 & \quad P_1(x)f=P_2(U^{-1}\circ V(x))f, \ \ x\in \Omega_1 \\
 \iff & \quad u_f^1(x)=u_f^2(U^{-1}\circ V(x)), \ \ x\in \Omega_1  \\
 \iff & \quad u_f^1\circ V^{-1}=u_f^2\circ U^{-1} \mbox{ on } V(\Omega_1)\subset \R^n.
 \end{align*}
 Since we have proven the latter, and since the boundary function $f$ was arbitrary, we have proven~\eqref{to_be_shown}. Consequently we have 
 \[
  P_1(\Omega_1)=P_2(U^{-1}(V(\Omega_1)))\subset P_2(\Omega_2)\subset P_2(M_2^\Gamma).
 \]
Thus $B$ extends to a neighborhood of the point $x_1\in \p B\setminus B$, which gives a contradiction. We have now proved that $B$ is closed. Since $M_1^\Gamma$ is connected, we conclude that $P_1(M_1^\Gamma)\subset P_2(M_2^\Gamma)$. We thus also have $J(M_1^\Gamma)\subset M_2^\Gamma$.
 
 Inverting the role of $(M_1,g_1)$ and $(M_2,g_2)$, and replacing $F$ and $J$ in the statement of the theorem by $F^{-1}$ and $J^{-1}$, shows that $P_1(M_1^\Gamma)= P_2(M_2^\Gamma)$ and that $J$ is surjective onto $M_2^\Gamma$. Consequently $J:M_1^\Gamma\to M_2^\Gamma$ is diffeomorphism by Lemma~\ref{J_is_local_dif}. Since we already showed that $J$ is a global harmonic morphism at the beginning of the proof in~\eqref{J_is_harm_morph}, the claim follows.
\end{proof}

Combining Lemma \ref{local_determination_on_B} and Theorem \ref{global_determination_of_harmonic_functions}, we have proved the following statement.

\begin{Theorem}
Let $(M_1,g_1)$ and $(M_2,g_2)$ be compact real analytic Riemannian manifolds, $n\geq 3$, with mutual boundary whose DN maps agree on an open set $\Gamma \subset \p M$. Assume also that $\Gamma$ is real analytic. Then there is a diffeomorphic (global) harmonic morphism $J:M_1^\Gamma\to M_2^\Gamma$ such that $J$ is real analytic in $M_1^{\Gamma}$ and $J|_{\Gamma}=\text{Id}$.
\end{Theorem}
\begin{proof}
We only need to show that $J$ is real analytic in $M_1^{\Gamma}$. In the interior this follows from the representation \eqref{representation_for_J} in terms of harmonic coordinates, which are real analytic in the interior. Near points of $\Gamma$ this follows from the statement of Lemma \ref{lemma_determination_near_the_boundary}, which implies that near points of $\Gamma$ one has $J = \psi_2^{-1} \circ \psi_1$ where $\psi_j$ are real analytic boundary normal coordinates.
\end{proof}

\section{Recovery of the Riemannian metric from a harmonic morphism} \label{sec_recovery_of_metric}

In the previous section we showed that the Poisson embedding can be used to determine the manifold up to a global harmonic morphism from the knowledge of the DN map in the real analytic case, $n\geq 3$. 
To give a new proof for the Calder\'on problem in the real analytic case, we need to show that a global harmonic morphism in this case is an isometry. Throughout this section, we assume that $(M_j,g_j)$, $j=1,2$, are compact 
connected and $C^\infty$ smooth, $n\geq 3$. 


We show that if $J=P_2^{-1}\circ P_1:M_1^\Gamma\to M_2^\Gamma$ is defined, and thus is a global harmonic morphism, it is then a homothety,
\[
J^*g_2=\lambda g_1, \quad \lambda \mbox{ constant},
\]
when $n\geq 3$. If we additionally assume that the DN-maps of $(M_1,g_1)$ and $(M_2,g_2)$ agree on $\Gamma\subset \p M$, then boundary determination implies that $\lambda = 1$.

It is known that a mapping between Riemannian manifolds having the same dimension $n \geq 3$ that pulls back \emph{local} harmonic functions to local harmonic functions is in fact a homothety~\cite{Fu78}, see also~\cite[Cor. 3.5.2]{Wood_book}. Our definition of harmonic morphisms assumes that the mapping pulls back \emph{global} harmonic functions to global harmonic functions. Our condition is seemingly slightly different, but it follows from the next result that, for manifolds having the same dimension, these conditions are equivalent.

We give a proof  that a global harmonic morphism is in fact a homothety when $n\geq 3$ by using harmonic coordinates. This seems to give a new proof for the result in the local case as well.

\begin{Proposition} \label{morphism_to_homothety}
 Let $(M_1,g_1)$ and $(M_2,g_2)$ be $C^\infty$ Riemannian manifolds having the same dimension $n\geq 3$ and having a mutual boundary $\partial M$. Let $\Gamma \subset \partial M$ be a nonempty open set, and let $J:(M_1^\Gamma,g_1)\to (M_2^\Gamma,g_2)$ be a locally diffeomorphic $C^\infty$ global harmonic morphism. Then $J$ is a homothety.
\end{Proposition}
\begin{proof}
 Let $x\in M_1^\Gamma$ and let $U=(u_{f_1}^2,\ldots, u_{f_n}^2)$ be an $n$-tuple of global harmonic functions that define harmonic coordinates on $\Omega_2$ near $J(x)\in M_2^\Gamma$, where $f_k\in C_c^\infty(\Gamma)$. This can be done by Proposition \ref{prop_runge_consequence_first} upon choosing $\{ d u_{f_1}^2(J(x)), \ldots, du_{f_n}^2(J(x)) \}$ linearly independent. Since $J$ is locally invertible, we have that $V=J^*U$ is a coordinate system near $x$ on $\Omega_1:=J^{-1}(\Omega_2)$. Since $J$ is a global harmonic morphism, the coordinate system $V$ is harmonic.
 
 In the coordinates $V$ and $U$ the coordinate representations of harmonic functions $u_f$ and $v_f=J^*u_f$ agree for arbitrary $f\in C_c^\infty(\Gamma)$. This is because
 \[
 u\circ U^{-1}=v\circ V^{-1} \iff v=u\circ U^{-1}\circ V = u\circ J =J^*u.
 \]
 In any $g_j$-harmonic coordinates, $j=1,2$, the Laplace-Beltrami operator is particularly simple:
 \[
 \Delta_{g_j}u=-g_j^{ab}\p_a\p_bu, \quad j=1,2.
 \]
 Here we abuse notation and denote by $g_1$ and $g_2$ the coordinate representations $(V^{-1})^*g_1$ and $(U^{-1})^*g_2$, respectively. Define $z:=V(x)=U(J(x))$ so $z \in \R^n$, and $\Omega:=V(\Omega_1)=U(\Omega_2)\subset\R^n$.
 
 For any symmetric matrix $(H_{ij})$ such that $g_2^{ij}(z)H_{ij}=0$ we may find by Runge approximation a global $g_2$-harmonic function $u = u_f$ with $f \in C^{\infty}_c(\Gamma)$ whose Hessian at $z$ is $(H_{ij})$ in the $g_2$-harmonic coordinates $U$. This is proved in Proposition~\ref{runge_prescribed_jet} (note that $\mathrm{Hess}_g(w)$ corresponds to $(\partial_{jk} w)$ in harmonic coordinates). 
 Thus we have that
 \begin{equation}\label{g_orthocomplement}
 \Tr{g_1(z)^{-1}H}=0=\Tr{g_2(z)^{-1}H},
 \end{equation}
 since the coordinate representations of the harmonic functions $v = v_f$ and $u = u_f$, and thus their Hessians, agree in the coordinates $V$ and $U$.
 
 Since $H$ can be any symmetric matrix with $g_2^{ij}(z)H_{ij}=0$, the above means that $g_1^{-1}$ and $g_2^{-1}$ have the same orthocomplement at $z$ with respect to the Hilbert-Schmidt inner product in the space of symmetric matrices. Thus $g_1^{-1}(z) = \lambda(z)^{-1}g_2^{-1}(z)$ for some nonzero real number $\lambda(z)$. Due to the positive definiteness of $g_j$, $j=1,2$, we have that $\lambda(z)>0$. The argument above can be repeated for all $z\in \Omega$, and we have
 \begin{equation}\label{comparable}
 g_1(z)=\lambda(z)g_2(z), \quad z\in \Omega,
 \end{equation}
 where $\lambda$ is a positive smooth function ($\lambda$ is smooth since $g_1$ and $g_2$ are). 
 
We next show that the function $\lambda$ is constant in $\Omega$. For this we use the fact that the coordinates, where~\eqref{comparable} holds, are harmonic. This is equivalent to saying that the contracted Christoffel symbols $g_j^{ab}\Gamma_{ab}^c(g_j)$ vanish. By lowering the index, this means that 
 \[
 \Gamma_a(g_j)=0,
 \]
 where $\Gamma_a(g_j)=-|g_j|^{-1/2}(g_j)_{ab}\p_c(|g_j|^{1/2}(g_j)^{bc})$.
 By taking the contracted Christoffel symbol of the equation~\ref{comparable}, we have
 \begin{equation}\label{Gamma_trick}
 0=\Gamma_a(g_1)=\Gamma_a(\lambda g_2)=\Gamma_a(g_2)-\frac{n-2}{2}\p_a\log \lambda=-\frac{n-2}{2}\p_a\log \lambda.
 \end{equation}
 Here we have used the following simple computation for the conformal scaling of the contracted Christoffel symbols:
  \begin{align*}
 \Gamma_a(\lambda g)&=-|\lambda g|^{-1/2}(\lambda g)_{ab}\p_c(|\lambda g|^{1/2}(\lambda g)^{bc})\\
 &=-\lambda^{-n/2+1}|g|^{-1/2}g_{ab}\p_c(\lambda^{n/2-1}|g|^{1/2}g^{bc}) \\
 &=\lambda^{-n/2+1}\left(\lambda^{n/2-1}\Gamma_a-\p_a\lambda^{n/2-1}\right) \\
 &=\Gamma_a-\lambda^{-n/2+1}\left(\frac{n}{2}-1\right)\lambda^{n/2-2}\p_a\lambda=\Gamma_a-\frac{n-2}{2}\p_a\log \lambda.
 \end{align*}
 Since $n\geq 3$, the equation~\ref{Gamma_trick} shows that $\lambda$ is constant in $\Omega$. Recalling that $g_1(z)$ and $g_2(z)$ were the coordinate representations of $g_1$ and $g_2$ in coordinates $V$ and $U$, with $V=J^*U$, we have $g_1=\lambda J^*g_2$ on $\Omega_1$. Since this identity holds near an arbitrary point $x\in M_1^\Gamma$ and since $M_1^{\Gamma}$ is connected, we have proved the claim.
\end{proof}

\begin{Remark}
We remark that in the setting of the proof above we can by equation~\ref{g_orthocomplement} actually express (a multiple of) $g_j^{-1}$ at $z$ in terms of Hessians of solutions $u_f^j$ at $z$ for some $f\in C_c^\infty(\Gamma)$, $j=1,2$. Let $H_j^k$, $k=1,\ldots,m$, $m=\frac{n(n+1)}{2}-1$, be a basis for the orthocomplement $\{g_j(z)^{-1}\}^\bot$ in the space of symmetric matrices equipped with the Hilbert-Schmidt inner product. By the Runge approximation of Proposition~\ref{runge_prescribed_jet}, we may find $f_k$ so that $H_k^j=\text{Hess}_{g_1}(u^1_{f_k}(z))=\text{Hess}_{g_2}(u^2_{f_k}(z))$.

Then we have
\begin{equation}\label{xpres_yourself}
g_j^{-1}(z) =\lambda_j \ast (H_j^1\wedge H_j^2\wedge \cdots \wedge H_j^m), \quad \lambda_j=\mbox{constant}\neq 0, \ j=1,2.
\end{equation}
Here $\ast$ is the Hodge star operator in the space of symmetric matrices defined with respect to the Hilbert-Schmidt inner product and $\wedge$ is the wedge product in that space. The equation~\eqref{xpres_yourself} holds because it follows from the definition of the Hodge star that the right hand side is orthogonal to each $H_j^k$. Thus the right hand side has the same orthocomplement as $g_j^{-1}(x_0)$ has.
\end{Remark}

Another remark is that since a homothety maps harmonic functions to harmonic functions, we have that a mapping between same dimensional Riemannian manifolds is global harmonic morphism if and only if it is a local harmonic morphism as defined in \cite[Definition 4.1.1]{Wood_book}.

Next we show that if the DN maps agree, the homothety constant $\lambda$ is $1$.
\begin{Proposition}\label{det_of_lambda_bndr}
 Assume the conditions of Proposition \ref{morphism_to_homothety}, and assume in addition that the DN maps of $(M_1,g_1)$ and $(M_2,g_2)$ agree on an open subset $\Gamma$ of the boundary. Also assume that $J|_{\Gamma} = \mathrm{Id}$. Then 
 \[
 J^*g_2=g_1.
 \]
\end{Proposition}
\begin{proof}
By Proposition \ref{morphism_to_homothety}, we know that $J$ is homothety and that 
\[
g_1 = \lambda J^* g_2 \text{ in $M_1^{\Gamma}$}, \qquad J|_{\Gamma} = \mathrm{Id}.
\]
Fix a point $p \in \Gamma$, and use boundary determination (see the proof of Lemma \ref{lemma_determination_near_the_boundary}) to deduce that 
\[
g_1 = \Psi^* g_2 \text{ in $U_1$}, \qquad \Psi|_{U_1 \cap \Gamma} = \mathrm{Id}
\]
for some diffeomorphism $\Psi$ defined in a neighborhood $U_1$ of $p$ in $M_1^{\Gamma}$.

Now if $v \in T_p(\partial M)$, the first equation and the fact that $J|_{\Gamma} = \mathrm{Id}$ imply that 
\[
g_1(v, v) = \lambda g_2(v,v),
\]
while the second equation gives that 
\[
g_1(v,v) = g_2(v,v).
\]
Thus $\lambda = 1$.
\end{proof} 

Combining the results so far, we have achieved a new proof of the uniqueness in the Calder\'on problem in the real analytic case when $n\geq 3$ \cite{LassasUhlmann}: \\[2pt]

\noindent \textbf{Theorem.} \emph{Let $(M_1,g_1)$ and $(M_2,g_2)$ be compact real analytic Riemannian manifolds, $n\geq 3$, with mutual boundary whose DN maps agree on an open set $\Gamma \subset \p M$. Assume also that $\Gamma$ is real analytic. Then there is a real analytic diffeomorphism $J:M_1^\Gamma\to M_2^\Gamma$ such that $g_1 = J^* g_2$ and $J|_{\Gamma}=Id$.}

\section{Uniqueness in the 2D Calder\'on problem} \label{sec_calderon_2d}

In this section we use the Poisson embedding technique to give a new proof of uniqueness in the Calder\'on problem in dimension $2$. This result is also due to~\cite{LassasUhlmann}. In this section we assume $(M_1,g_1)$ and $(M_2,g_2)$ are compact, connected $C^\infty$ Riemannian manifolds with mutual boundary $\partial M$. Note that it is not required that the manifolds are real analytic.

\begin{Theorem}[Uniqueness in the Calder\'on problem in 2D]\label{unique2d}
 Let $(M_1,g_1)$ and $(M_2,g_2)$ be two-dimensional compact connected $C^\infty$ Riemannian manifolds with mutual boundary $\partial M$. Assume that the DN maps agree on an open subset $\Gamma \subset \p M$. Then there is a conformal diffeomorphism $J:M_1^\Gamma\to M_2^\Gamma$ such that 
 \[
 J^*g_2=\lambda g_1.
 \]
 Here $\lambda$ is a smooth positive function in $M_1^{\Gamma}$, $\lambda|_{\Gamma}=1$, and $J|_\Gamma=\mathrm{Id}$.
\end{Theorem}

There is no assumption on real analyticity in this result. The proof relies on the fact that on two-dimensional manifolds there exist \emph{isothermal coordinates} near any point, i.e.\ coordinates $(u_1, u_2)$ such that $du_1 = * du_2$, see \cite[Section 5.10]{Taylor}. In these coordinates the metric looks like $g_{jk} = c \delta_{jk}$ for some positive function $c$, see Lemma \ref{isotherm_char}. Isothermal coordinates are also harmonic coordinates in dimension $2$. We will use both of these facts.

We first prove local determination of harmonic functions near a boundary point, and then extend local determination to global determination. These are analogues of Lemma~\ref{local_determination_on_B} and Theorem~\ref{global_determination_of_harmonic_functions}. After this, a two-dimensional version of Proposition~\ref{morphism_to_homothety} determines the metric up to a conformal mapping.

For the determination of harmonic functions near a boundary point, we note that in isothermal coordinates a $g$-harmonic function actually satisfies the Laplace equation in a subset $\R^2$. We show that the boundary determination result~\cite{LU} of the metric in boundary normal coordinates implies determination of the metric on the boundary also in isothermal coordinates. Determination of harmonic functions near the boundary then follows from unique continuation for harmonic functions on $\R^2$.

The determination of harmonic functions near the boundary in isothermal coordinates involves some technicalities. These are consequences of the fact that the boundary determination result of~\cite{LU}, that we rely on, is given in boundary normal coordinates instead of isothermal coordinates. We address the technicalities in the next lemma, whose proof is given in the Appendix~\ref{2D_appx}.
\begin{Lemma}\label{bndr_determination_sothermal_coordinates}
  Let $(M_1,g_1)$ and $(M_2,g_2)$ be two-dimensional Riemannian manifolds with boundary whose DN maps agree on an open subset $\Gamma \subset \p M$. Then for any $p\in \Gamma$ there are isothermal coordinates $U_j$, $j=1,2$, defined on neighborhoods $\Omega_j\subset M_j^\Gamma$ of $p$ such that the following statements hold:
  \begin{enumerate}
   \item There is an open subset $\Gamma_0$ of $\Gamma$ with $p \in \Gamma_0$ such that $U_1(\Gamma_0)=U_2(\Gamma_0)=:\tilde{\Gamma}\subset \R^2$ and $U_1|_{\Gamma_0} = U_2|_{\Gamma_0}$.
   \item If $f\in C_c^\infty(\Gamma)$, then the Cauchy data of the coordinate representations $U_1^{-1*}u_f^1$ and $U_2^{-1*}u_f^2$ agree on $\tilde{\Gamma}\subset \R^2$.
  \end{enumerate}
\end{Lemma}

\begin{Lemma}[Near boundary determination in 2D]\label{det_near_bndr_2d}
Assume the conditions in the previous lemma. Then for any $p\in \Gamma$ there are  
 isothermal coordinates $U_j$, $j=1,2$, defined on neighborhoods $\Omega_j \subset M_j^{\Gamma}$ of $p$ 
 such that for $f\in C_c^\infty(\Gamma)$, we have
 \[
 u_f^1\circ U_1^{-1}(x)=u_f^2\circ U_2^{-1}(x), \quad x\in U_1(\Omega_1)\cap U_2(\Omega_2)\subset \{ x^2 \geq 0 \}.
 \]
 Moreover, there exists an open set $B \subset M_1^\Gamma$ with $\Gamma \subset B$ and a $C^{\infty}$ diffeomorphism $F: B \to F(B) \subset M_2^\Gamma$ such that 
 \[
  F(B) \subset M_2^\Gamma, \qquad  P_1 = P_2 \circ F \text{ on $B$}.
  \]
\end{Lemma}

\begin{proof}
Let $f\in C_c^\infty(\Gamma)$ and let $u_f^j$, $j=1,2$, be the corresponding harmonic functions on $(M_j,g_j)$. Let $U_j$ be the isothermal coordinates of the previous lemma. Then the Cauchy data of $u_f^j$ agree on $\tilde{\Gamma}\subset \R^2$ in coordinates $U_j$. In isothermal coordinates, which are always also harmonic coordinates in dimension $2$, the Laplace-Beltrami equation for $u_f^j$ reads
\[
c_1^{-1}\Delta_{\R^n}(u_f^1\circ U_1^{-1})=0=c_2^{-1}\Delta_{\R^n}(u_f^2\circ U_2^{-1}).
  \]
  Thus we see that $u_f^j\circ U_j^{-1}$ satisfy the same Euclidean Laplace equation with the same Cauchy data locally on a smooth mutual part of the boundary of the domain. By elliptic unique continuation, see e.g.~\cite[Theorem 3.3.1]{Isakov_book}, and by setting $F=U_2^{-1}\circ U_1$ we obtain the claim with $B$ replaced by $U_1^{-1}(U_1(\Omega_1)\cap U_2(\Omega_2))$. We can then enlarge $B$ as in Lemma \ref{local_determination_on_B} to conclude the proof.
  \end{proof}

We record the following:
\begin{Proposition}\label{morphism_to_homothety2d}
 Let $(M_1,g_1)$ and $(M_2,g_2)$ be two-dimensional $C^\infty$ Riemannian manifolds with mutual boundary. Let $J: M_1^{\Gamma} \to M_2^\Gamma$ be a locally diffeomorphic $C^\infty$ global harmonic morphism. Then $J$ is conformal,
 \[
  J^*g_2=\lambda g_1 \text{ in $M_1^{\Gamma}$}
 \]
 for some positive function $\lambda \in C^{\infty}(M_1^{\Gamma})$.
\end{Proposition}
The proof is identical to that of Proposition~\ref{morphism_to_homothety} except that we cannot deduce that $\lambda(x)$ is constant by the argument using equation \eqref{Gamma_trick}. We omit the proof.

We will now prove global determination of harmonic functions.

\begin{Theorem}\label{global_determination_of_harmonic_functions_2D}
Let $(M_1,g_1)$ and $(M_2,g_2)$ be compact $2$-dimensional $C^\infty$ smooth Riemannian manifolds with mutual boundary. Let $\Gamma$ be an open subset of $\p M$. Assume that there is a neighborhood $B\subset M_1^\Gamma$ of a boundary point $p\in \Gamma$ and a mapping $F:B\to F(B)\subset M_2^\Gamma$ diffeomorphic onto its image such that
 \[
 P_1=P_2\circ F.
 \]
 Then we have $P_1(M_1^\Gamma)=P_2(M_2^\Gamma)$ and 
 \[
 J=P_2^{-1}\circ P_1:M_1^\Gamma\to M_2^\Gamma
 \]
 is $C^\infty$ diffeomorphic global harmonic morphism extending $F$.
\end{Theorem}
\begin{proof}
We proceed as in the proof of Theorem~\ref{global_determination_of_harmonic_functions} to which we refer the reader for more details. Let us recall the notation and some facts from there. We redefine $\emptyset\neq B\subset M_1^\Gamma$ to be the largest open connected set such that $P_1(B)\subset P_2(M_2^\Gamma)$. The task is to show that $B$ is closed. We argue by contradiction and assume that it is not. Then the points $x_1\in \p B\setminus B$ and $x_2\in M_2^\Gamma$ are limits of sequences $(p_k)\subset B$ and $(J(p_k))\subset J(B)$. If $f\in C_c^\infty(\Gamma)$, we have $u_f^1=u_f^2\circ J$ on $B$.

We construct isothermal coordinates on neighborhoods $\Omega_1$ of $x_1\in M_1^\Gamma$ and $\Omega_2$ of $x_2\in M_2^\Gamma$ as follows. Let us first choose by Runge approximation a boundary function $f\in C_c^\infty(\Gamma)$ such that
\[
 du_{f}^2(x_2)\neq 0.
\]
Let us denote by $u_1$ a harmonic conjugate of $u_f^2$ near $x_2$. This is a function solving near $x_2$ the equation
\[
du_1 = -\ast_{g_2} d u_f^2,
\]
where $\ast_{g_2}$ is the Hodge star of $g_2$. A local solution $u_1$ exists since the right hand side is closed, because $u_f^2$ is harmonic. We may assume $u_1(x_2)=0$. Let us denote
\[
U=(u_1,u_f^2).
\]
Then $U$ is an isothermal coordinate system on a neighborhood $\Omega_2$ of $x_2$, see Lemma~\ref{isotherm_char}.

Likewise, let $v_1$ be a harmonic conjugate of $u_f^1$ near $x_1$ with $v_1(x_1)=0$. Thus $v_1$ solves
\[
 dv_1=-\ast_{g_1} d u_f^1.
\]
We set
\[
 V=(v_1,u_f^1).
\]
These are isothermal coordinates on a neighborhood $\Omega_1$ of $x_1$. The fact that the Jacobian of $V$ is invertible at $x_1$ follows from
\[
 du_f^1(x_1)=\lim du_f^1(p_k)=\lim du_f^2(J(p_k))=du_f^2(x_2)\neq 0.
\]
Redefine $\Omega_1$, if necessary, so that $V(\Omega_1)\subset U(\Omega_2)$.

We next show that on $B\cap \Omega_1$ we have $v_1=u_1\circ J$. By using $u_f^1=u_f^2\circ J$ on $B$ we have
\[
 - dv_1=\ast_{g_1} d u_f^1=\ast_{g_1} d (J^*u_f^2)=\ast_{g_1}J^*du_f^2.
\]
By Lemma~\ref{J_is_local_dif} $J$ is a $C^{\infty}$ diffeomorphism in $B$. By Proposition~\ref{morphism_to_homothety2d} applied with $M_1^{\Gamma}$ replaced by $B$ and $M_2^{\Gamma}$ replaced by $J(B)$ (the proof of Proposition \ref{morphism_to_homothety2d} is really a pointwise argument and applies in this case), we have that $J$ is a conformal mapping on $B$, $J^*g_2=\lambda g_1$. Thus we have
\[
\ast_{g_1}J^*du_f^2=J^*(\ast_{(g_2/(\lambda\circ J^{-1}))}du_f^2).
\]
Since Hodge star is conformally invariant when operating on $1$-forms in dimension $2$, the above is
\[
J^*(\ast_{g_2}du_f^2)=-J^*du_1=-d(u_1\circ J).
\]
Thus $dv_1=d(u_1\circ J)$ on $\Omega_1\cap B$. Since $v_1(x_1)=u_1(x_2)=0$, we have $v_1=u_1\circ J$ on $\Omega_1\cap B$. Consequently on $\Omega_1\cap B$ we have 
\[
 V=J^*U, \text{ or equivalently } J=U^{-1}\circ V.
\]
(The point here is that $U^{-1}\circ V$ is defined on the whole $\Omega_1$ and gives us a good candidate for a local extension of $J$ onto the whole $\Omega_1$.)

 Let $f\in C_c^{\infty}(\Gamma)$. To conclude the proof, we will show that on $V(\Omega_1)$, we have
 \[
 u_f^1\circ V^{-1}=u_f^2\circ U^{-1}.
 \]
 Since $u_f^1=u_f^2\circ J$ on $B$, for $f\in C_c^\infty(\Gamma)$, the above holds on the open set $V(\Omega_1\cap B)$.
 Since the coordinates in question are isothermal and harmonic (where $\Gamma^a(g_j)=0$) we have that
 \[
 c_1^{-1}\Delta_{\R^2}(u_f^1\circ V^{-1})=0=c_2^{-1}\Delta_{\R^2}(u_f^2\circ U^{-1}).
 \]
 Thus $u_f^1\circ V^{-1}$ and $u_f^2\circ U^{-1}$ both satisfy the Laplace equation in $V(\Omega_1) \subset \R^2$. Since these functions agree on the open set $V(\Omega_1\cap B)$ they agree everywhere on $V(\Omega_1)$ by unique continuation. This shows that we may indeed extend $J$ to $B \cup \Omega_1$, which gives a contradiction and concludes the proof.
\end{proof}
\begin{proof}[Proof of Theorem~\ref{unique2d}]
 By Lemma~\ref{det_near_bndr_2d}, we have that there is $B\subset M_1^\Gamma$ and a diffeomorphic harmonic morphism $F:B\to F(B)\subset M_2^\Gamma$. By Theorem~\ref{global_determination_of_harmonic_functions_2D} the mapping $F$ extends to a global harmonic morphism $J:M_1^\Gamma\to M_2^\Gamma$. Proposition~\ref{morphism_to_homothety2d} shows that $J$ is a conformal mapping. That the implied conformal factor is $1$ on $\Gamma$ follows from calculations in the proof of Proposition~\ref{det_of_lambda_bndr}. 
\end{proof}


\section{On determining the coefficients of quasilinear elliptic operators from source-to-solution mapping}\label{sec:q_lin}
In this section we apply the Poisson embedding technique for a Calder\'on type inverse problem for second order quasilinear elliptic operators on Riemannian manifolds. 
Let $(M,g)$ be a Riemannian manifold with boundary. The quasilinear operators
\[
Q:C^{\infty}(M)\to C^{\infty}(M)
\]
we study are assumed to have the coordinate representation
\begin{equation}\label{quasilinear_op_form}
Qu(x)=\mathcal{A}^{ab}(x,u(x),du(x))\nabla_a\nabla_bu(x)+\mathcal{B}(x,u(x),du(x)).
\end{equation}
Here we assume that $\mathcal{A}^{ab}(x,c,\sigma)\in T_{x}^2M$ for given $(x,c,\sigma)\in M\times \R\times T_x^*M$ is a $2$-tensor field and $\mathcal{B}(x,c,\sigma)$ is a function. The covariant derivative $\nabla$ is determined by $g$ as usual. We consider $\mathcal{A}$ and $\mathcal{B}$ as mappings
\begin{equation}\label{elliptic_a}
\mathcal{A}: M\times \R\otimes T^*M \to T^2_0(M)
\end{equation}
and 
\begin{equation}\label{elliptic_b}
\mathcal{B}: M\times \R\otimes T^*M \to \R
\end{equation}
where $\pi(\mathcal{A}(x,c,\sigma))=x$. Here $\pi$ is the canonical projection $T^2_0(M)\to M$ for $2$-contravariant tensors. The notation $M\times \R\otimes T^*M$ refers to the subset
\[
\{(x,c,\sigma)\in M\times \R\times T_x^*M\}
\]
of $M\times\R\times T^*M$. We can also think of $M\times \R\otimes T^*M$ as tensor product of the trivial line bundle $M\times \R$ and the vector bundle $T^*M$.

We assume that $Q$ is quasilinear elliptic, which means that for all $(x,c,\sigma)\in M\times \R\otimes T^*M$ and $\xi\in T_x^*M$ we have
\begin{equation}\label{elliptic}
 \mathcal{A}^{ab}(x,c,\sigma)\xi_a\xi_b\geq\lambda |\xi|^2_g, \quad \lambda >0.
\end{equation}
We assume that the coefficients $\mathcal{A}$ and $\mathcal{B}$ are $C^\infty$ smooth and in the main theorem of this section, Theorem~\ref{qlin_main_thm}, we assume that they are real analytic. We will assume throughout that $0$ is a solution, i.e.\ $Q(u) = 0$, which is equivalent with the condition 
\begin{equation} \label{b_assumption}
\mB(x,0,0) = 0.
\end{equation}
The linearization of $Q$ at $u=0$ is the operator 
\begin{equation} \label{q_linearization}
Lu = \mA^{ab}(x,0,0) \nabla_a \nabla_b u + \frac{\partial \mB}{\partial u}(x,0,0) u + \frac{\partial \mB}{\partial \sigma_j}(x,0,0) \partial_j u
\end{equation}
where in the last term we have identified $T_x^* M$ with $\mR^n$. We will also assume that, for some fixed $\alpha$ with $0 < \alpha < 1$,  
\begin{equation} \label{l_assumption}
L: C^{2,\alpha}(M) \cap H^1_0(M) \to C^{\alpha}(M) \text{ is invertible}.
\end{equation}
It follows from these assumptions that there are unique small solutions of $Q(u) = f$, $u|_{\partial M} = 0$ when $f$ is small. The proof of the next standard result is given in Appendix \ref{sec_appendix}.

\begin{Proposition} \label{prop_nonlinear_wellposedness}
Let $(M,g)$ be a compact manifold with smooth boundary, let $\mathcal{A}$, $\mathcal{B}$ be $C^{\infty}$ maps satisfying \eqref{elliptic_a}--\eqref{b_assumption}, and let 
\[
Q(u) = \mA^{ab}(x, u, du) \nabla_a \nabla_b u + \mB(x,u, du).
\]
Assume that $L$ is the linearization of $Q$ at $u=0$ given in \eqref{q_linearization}, and assume that \eqref{l_assumption} holds.

There are constants $C, \eps, \delta > 0$ such that whenever $\norm{f}_{C^{\alpha}(M)} \leq \eps$, the equation $Q(u) = f$ in $M$ with $u|_{\partial M} = 0$ has a solution $u \in C^{2,\alpha}(M)$ satisfying $\norm{u}_{C^{2,\alpha}(M)} \leq C \norm{f}_{C^{\alpha}(M)}$. If $u_j\in C^{2,\alpha}(M)$, $j=1,2$, both solve $Q(u_j) = f$ in $M$ with $u_j|_{\partial M} = 0$ and $\norm{u_j}_{C^{2,\alpha}(M)} \leq \delta$, then $u_1=u_2$.
\end{Proposition}


Operators of the above form appear e.g.\ in the study of minimal surfaces or prescribed scalar curvature questions (Yamabe problem), see \cite{GilbargTrudinger, Taylor} for more information.

We describe the inverse problem we are about to study. Let $Q, \eps, \delta$ be as in Proposition \ref{prop_nonlinear_wellposedness}. We assume that we know the \emph{source-to-solution mapping} of $Q$ on an open subset $W$ of $M$. The source-to-solution mapping 
\begin{equation} \label{s_w_definitiona}
S: \{ f \in C_c^\infty(W) \,;\, \norm{f}_{C^{\alpha}(M)} \leq \eps \} \to C^\infty(W)
\end{equation}
is defined as
\begin{equation} \label{s_w_definition}
S: f\mapsto u|_W,
\end{equation}
where $u$ is the unique solution to
\begin{equation} \label{s_w_definitionb}
Qu=f \mbox{ on } M, \qquad u=0 \mbox{ on } \p M, \qquad \norm{u}_{C^{2,\alpha}(M)} \leq \delta.
\end{equation}
Note that indeed $u \in C^{\infty}(M)$, using Schauder estimates for linear elliptic equations and the fact that $u \in C^{2,\alpha}(M)$. The aim of the inverse problem is to determine the coefficients $\mathcal{A}$ and $\mathcal{B}$ of $Q$ up to a diffeomorphism and possible other symmetries of the problem. When the coefficients of $Q$ are real analytic, our main theorem shows that in this case there is only one additional symmetry, which we describe next.  

We note that there is a simple transformation between coefficient $\mathcal{B}$ and the Christoffel symbols $\Gamma_{ab}^k$ contracted by $\mathcal{A}^{ab}$ that leaves the source-to-solution mapping intact: Assume that $u$ solves
\begin{equation}\label{gauge_sym_begin}
Qu=f \text{ on } M.
\end{equation}
Then $u$ also solves
\[
\tilde{Q}u=f,
\]
where 
\[
\tilde{Q}u=\mathcal{A}^{ab}(x,u(x),du(x))\tilde{\nabla}_a\tilde{\nabla}_bu+\tilde{\mathcal{B}}(x,u(x),du(x))
\]
and where $\tilde{\mathcal{B}}$ is defined as
\begin{equation}\label{Q_gauge}
\tilde{\mathcal{B}}(x,c,\sigma)=\mathcal{B}(x,c,\sigma)+\mathcal{A}(x,c,\sigma)^{ab}(\tilde{\Gamma}_{ab}^k(x)-\Gamma_{ab}^k(x))\sigma_k.
\end{equation}
Here $\tilde{\nabla}$ and $\tilde{\Gamma}_{ab}^k$ denote the Levi-Civita connection and Christoffel symbols of some other Riemannian metric $\tilde{g}$ on $M$. 
Therefore the source-to-solution mapping defined with respect to $\tilde{Q}$ coincides with the source-to-solution map $S$ of $Q$. Note that even though Christoffel symbols does not constitute a tensor field, the difference of two Christoffel symbols $\tilde{\Gamma}_{ab}^k-\Gamma_{ab}^k$ is a tensor field. It follows that we can not make the symmetry vanish by choosing a suitable coordinate system. This symmetry will be called the gauge symmetry of the inverse problem.

The gauge symmetry is an obstruction for finding $\mathcal{B}$, $\mathcal{A}$ and $\Gamma_{ab}^k$ independently of each other in the general case. However, in some cases when we have extra information about the coefficients $\mathcal{A}$ and $\mathcal{B}$, the gauge symmetry vanishes. We give examples of conditions when this happens in Corollary~\ref{qlin_cor_resolve_thm}.

We remark that if the coefficients are not real analytic, other symmetries in the inverse problem can appear. An easy example is the standard Laplace-Beltrami operator in dimension $2$ where one can scale the metric by a positive function that is constant $1$ on the measurement set $W$ without affecting the source-to-solution mapping. Another similar example is given by the conformal Laplacian in dimensions $n\geq 3$~\cite{LLS}.

Our main theorem of this section is the following determination result.
\begin{Theorem}\label{qlin_main_thm}
 Let $(M_1,g_1)$ and $(M_2,g_2)$ be compact connected real analytic Riemannian manifolds with mutual boundary and assume that $Q_j$, $j=1,2$, are quasilinear operators of the form~\eqref{quasilinear_op_form} having coefficients $\mathcal{A}_j$, $\mathcal{B}_j$ satisfying \eqref{elliptic_a}--\eqref{l_assumption}. Moreover, assume that $\mathcal{A}_j$ and $\mathcal{B}_j$ are real analytic in all their arguments. 
 
 Let $W_j$, $j=1,2$, be open subsets of $M_j$, and assume that there is a diffeomorphism $\phi:W_1\to W_2$ so that the source-to-solutions maps $S_j$ for $Q_j$ agree in the sense that
\[
\phi^*S_2f=S_1\phi^*f,
\]
for all $f\in C_c^\infty(W_2)$ with $\norm{f}_{C^{\alpha}(M_2)}$ sufficiently small.

Then there is a real analytic diffeomorphism $J:\Mi_1\to \Mi_2$ such that
\[
\mathcal{A}_1=J^*\mathcal{A}_2=:\mathcal{A}
\]
and 
\[
\mathcal{B}_1-J^*\mathcal{B}_2=\mathcal{A}^{ab}(\Gamma(g_1)_{ab}^k-\Gamma(J^*g_2)_{ab}^k)\sigma_k.
\]
The mapping $J$ satisfies
\[
J|_{\Wi_1}=\phi: \Wi_1\to \Wi_2
\]
where $\Wi_1=W_1\cap \Mi_1 \text{ and } \Wi_2=W_2\cap \Mi_2$.
\end{Theorem}

The assumption $\phi^*S_2f=S_1\phi^*f$, for $f\in C_c^\infty(W_2)$ small, means that the diagram
\[
\xymatrix{
 C_c^\infty(W_2) \ar[d]^{\phi^*} \ar[r]^{S_2} &C^\infty(W_2)\ar[d]^{\phi^*}\\
 C_c^\infty(W_1) \ar[r]^{S_1}
 &C^\infty(W_1)}
\]
commutes when $f$ is small.

We describe our strategy for proving the theorem. By the arguments in the preceding sections, it would be natural to define a mapping analogous to the Poisson embedding for the quasilinear elliptic operator $Q$ and then use tools analogous to those we built around the Poisson embedding. 
However, as far as we know, Runge approximation for quasilinear operators is not known. This prevents us of using this natural approach for the moment.

Instead we do the following. We linearize the source-to-solution mapping (at the mutual solution $0$) that yields a linear Calder\'on type inverse problem for a linear second order elliptic operator whose source-to-solution map is known. For this linearized problem we use the Poisson embedding technique modified slightly to deal with the source-to-solution map instead of the DN map. In this way we will find the manifold up to a real analytic diffeomorphism. This is the first step. The modified Poisson embedding is given in Definition~\ref{Poisson_embed_source_map}.

The second step is the following. We will see that knowing the source-to-solution map on the open set $W$ determines the coefficients $\mathcal{A}$ and $\mathcal{B}$ on $W$, up to the gauge symmetry described in~\eqref{Q_gauge}. In this step we read the coefficients $\mathcal{A}$ and $\mathcal{B}$ in $W$ from the solutions, which is similar to the argument in Proposition~\ref{morphism_to_homothety}.

Since we have determined the manifold up to a real analytic diffeomorphism, we can view the coefficients $\mathcal{A}$ and $\mathcal{B}$ on a single fixed manifold and use standard real analytic unique continuation there. This determines the coefficients of the quasilinear operator on the whole manifold up to a diffeomorphism and the gauge symmetry.

As already mentioned, with some suitable extra information about the coefficients $\mathcal{B}$ and $\mathcal{A}$ the gauge symmetry vanishes and we can determine $\mathcal{A}$ and $\mathcal{B}$ independently. 
\begin{Corollary}\label{qlin_cor_resolve_thm}
 Assume the conditions and notation in Theorem~\ref{qlin_main_thm}, and assume also one of the following:
 \begin{enumerate}
\item[(1)]
$\mathcal{A}_1(x,c,\sigma)$ (or $\mathcal{A}_2(x,c,\sigma)$) is $s$-homogeneous in the $\sigma$-variable and $\mathcal{B}_1(x,c,\sigma)$ and $\mathcal{B}_2(x,c,\sigma)$ are $s'$-homogeneous with $s'\neq s+1$ for all $x\in W_1$ and $c\in \mathbb{R}$; or
 
 %
 
 
\item[(2)]
$\phi:(\Wi_1,g_1|_{\Wi_1})\to (\Wi_2,g_2|_{\Wi_2})$ is an isometry.
\end{enumerate}
 Then we have
 \[
\mathcal{A}_1=J^*\mathcal{A}_2 \mbox{ and } \mathcal{B}_1=J^*\mathcal{B}_2
 \]
 and also
 \[
 \mathcal{A}_1^{ab}\Gamma_{ab}^k(g_1)=(J^*\mathcal{A}_2)^{ab}\Gamma_{ab}^k(J^*g_2).
 \]
\end{Corollary}
Note that the corollary does not claim that we can find the Riemannian metrics $g_1$ and $g_2$ up to $J$. An example satisfying the first condition is the  nonlinear Schr\"odinger operator
\[
 u\in C^\infty \mapsto -\Delta_gu+q|u|^2u, \quad q\in C^\infty.
\]
An example satisfying the second condition is the case, where one knows the Riemannian metric and the manifold ($\phi=\text{Id}$) on a measurement set $W$.

 \subsection{Linearized problem}
 Let us first linearize the source-to-solution map of the quasilinear problem. This yields a Calder\'on type inverse problem for the linearized equation. The proof of the following result is given in Appendix \ref{sec_appendix}.

\begin{Proposition} \label{prop_nonlinear_linearization}
Let $(M,g)$ and $Q$ be as in Proposition \ref{prop_nonlinear_wellposedness}. Let $W \subset M$ be open, and let $S$ be the source-to-solution map defined in \eqref{s_w_definition}. Then, for any $f \in C^{\infty}_c(W)$, 
\[
\lim_{t \to 0} \frac{S(tf) - S(0)}{t} = S^L(f) \quad \text{(limit in $C^1(\overline{W})$)}
\]
where $S^L: C^{\infty}_c(W) \to C^{\infty}(W), \ f \mapsto u$ is the source-to-solution map of the linearized equation $Lu = f$ in $M$ with $u|_{\partial M} = 0$, where $L$ is given in \eqref{q_linearization}.
\end{Proposition}

We have now seen that linearizing the source-to-solution mapping $S$ of a quasilinear equation leads to a Calder\'on type inverse problem for a linear equation. Next we show that in the real analytic case the source-to-solution mapping of the linearized problem determines the manifold up to a real analytic diffeomorphism. Precisely we will prove: 
\begin{Theorem}\label{global_determination_of_harmonic_functions_for_L}
Let $(M_1, g_1)$ and $(M_2, g_2)$ be compact real analytic manifolds with mutual boundary. Let $L_1$ and $L_2$ be second order uniformly elliptic partial differential operators on $M_1$ and $M_2$, respectively, of the general form
 \begin{equation}\label{L_form}
 L_j=A_j^{ab}(x)\nabla_a\nabla_bu(x)+B_j^a(x)\nabla_au(x)+C_j(x)u(x),
 \end{equation}
 where $A_j^{ab}$ and $B_j^a$ are the components of a $2$-tensor field $A$ and vector field $B_j$, and $C$ is a function. Assume that $L_j$ are injective on $C^{\infty}(M_j) \cap H^1_0(M_j)$, and assume that $A$, $B$ and $C$ are real analytic up to boundary.
 
Let $W_j\subset M_j$ be open sets, and let $\phi:W_1\to W_2$ be a diffeomorphism so the source-to-solutions maps $S^L_j$ of $L_j$ satisfy 
\[
S^L_1\phi^*f=\phi^*S^L_2f, \quad f\in C_c^\infty(W_2).
\]
Then there is a real analytic diffeomorphism $J:\Mi_1\to \Mi_2$ with $J|_{\Wi_1}=\phi|_{\Wi_1}: \Wi_1\to \Wi_2$. If $f\in C_c^\infty(W_2)$, then $u_{f\circ \phi}^1=J^*u_f^2$. 
\end{Theorem}
In the theorem $u_{f\circ\phi}^1$ and $u_f^2$ are the solutions to $L_1u_{f\circ\phi}^1=f\circ\phi$ in $M_1$ and to $L_2u_{f}^2=f$ in $M_2$ with $u_{f\circ\phi}^1|_{\p M}=u_{f}^2|_{\p M}=0$.

We prove the theorem by modifying the Poisson embedding technique to suit the study of source-to-solution map instead of the DN map. The arguments are very similar to those that we have used in the previous sections. We keep the exposition short.

\begin{Definition}[Poisson embedding for $L$]\label{Poisson_embed_source_map}
 Let $(M,g)$ be a compact Riemannian manifold with boundary, and let $W$ be an open subset of $M$. Let $L$ be a second order elliptic differential operator of the form~\eqref{L_form} which is injective on $C^{\infty}(M) \cap H^1_0(M)$. The \emph{Poisson embedding} $R$ of the manifold $M$ is defined to be the mapping
 \[
 R:M^{\mathrm{int}}\to \mathcal{D}'(W)
 \]
 such that $R(x)f=u_f(x)$, where $u_f$ solves the Poisson problem
 \begin{align*}
 Lu_f&=f \mbox{ in } M, \\
 u_f&=0 \mbox{ on } \p M.
 \end{align*}
\end{Definition}
In the definition $\mathcal{D}'(W)$ means $[C_c^\infty(W)]^*$ as usual.
The reason why we consider $R$ to be defined only in the interior of $M$ is because we assume that the boundary values of the solutions $u_f$ vanish on the boundary. Thus we have no control on the points on the boundary by using solutions $u_f$. Even though $R$ is not the Poisson embedding of the previous section, we use the same name for $R$, and we note that $R$ is related to the linear elliptic operator $L$.

The basic properties of the Poisson embedding $R$ are as follows. The proof of the following proposition is in Appendix~\ref{proofs_of_smoothness}.
\begin{Proposition}\label{R_basics}
 Let $(M,g)$ be smooth compact manifold with boundary. For any $x \in \Mi$, one has $R(x) \in H^{-s}(W)$ whenever $s+2>n/2$. The mapping $R$ is continuous $\Mi \to H^{-s-1}(W)$ and $k$ times Fr\'echet differentiable considered as a mapping $\Mi \to H^{-s-1-k}(W)$. In particular, 
 $R:\Mi \to \mathcal{D}'(W)$ is $C^\infty$ smooth in the Fr\'echet sense. The mapping $R$ can be extended continuously to a mapping $M\to H^{-s}(W)$ by defining $R|_{\p M}=0$.
 
 The Fr\'echet derivative of $R$ at $x\in\Mi$ is a linear mapping given by
 \[
 DR_x:T_x\Mi\to \mathcal{D}'(W), \quad (DR_xV)f=du_f(x)\cdot V,
 \]
 where $u_f$ solves $Lu_f=f$ in $M$, $u_f|_{\p M}=0$, $f\in C_c^{\infty}(W)$, and $\cdot$ refers to the canonical pairing of vectors and covectors on $M$.
\end{Proposition}

In the statement, we are not claiming that the continuation of $R$ onto $M$ is injective on $\p M$.

To prove that $R$ is an embedding, we use the following analogue of Proposition \ref{prop_runge_consequence_first} which follows from a suitable Runge approximation result. Its proof is in Appendix~\ref{runge_apprx_sec}.

\begin{Proposition} \label{prop_runge_consequence_interior_source_sec6}
Let $(M,g)$ be a compact manifold with boundary, and let $L$ be a second order uniformly elliptic differential operator on $M$ which is injective on $C^{\infty}(M) \cap H^1_0(M)$. Let $W$ be a nonempty open subset of $M$, and denote by $u_f$ the solution of $Lu = f$ in $M$ with $u|_{\partial M} = 0$.
\begin{enumerate}
\item[(a)] 
If $x \in \Mi$, $y \in M$ and $x \neq y$, there is $f \in C^{\infty}_c(W)$ such that 
\[
u_f(x) \neq u_f(y).
\]
\item[(b)] 
If $x \in \Mi$ and $v \in T_x^* M$, there is $f \in C^{\infty}_c(W)$ such that 
\[
du_f(x) = v.
\]
\end{enumerate}
\end{Proposition}

\begin{Proposition}[$R$ is an embedding]\label{R_is_embedding}
  Let $M$, $L$, $R$ be as in Definition \ref{Poisson_embed_source_map}. The mapping $R:\Mi\to \mathcal{D}'(W)$ is a $C^\infty$ embedding (and a $C^k$ embedding $\Mi\to H^{-s-1-k}(W)$) in the sense that it is injective with injective Fr\'echet differential on $T\Mi$.
\end{Proposition}
\begin{proof}
 The injectivity of $R$ follows from Proposition \ref{prop_runge_consequence_interior_source_sec6}(a). Let $x\in \Mi$ and $V\in T_xM$. Assume that $0=(DR_xV)f=du_f(x)\cdot V$ for all $f\in C_c^\infty(W)$. By Proposition \ref{prop_runge_consequence_interior_source_sec6}(b) one can find $f\in C_c^\infty(W)$ so that $du_f(x)\cdot V\neq 0$, unless $V=0$. This shows injectivity of the differential.
\end{proof}
We construct next local coordinate systems from solutions $u_f$ to $Lu_f=f$, $u_f|_{\p M}=0$. We call these coordinates \emph{solution coordinates}. These coordinates are constructed by Runge approximating local solutions to $Lu=0$. 
\begin{Lemma}[Solution coordinates]\label{solution_coords}
 Let $(M,g)$ be a compact Riemannian manifold with boundary. Let $W$ be an open subset of $M$ and let $x_0\in \Mi$. Then there is $C^\infty$ coordinate system on a neighborhood $\Omega$ of $x_0$ of the form $(u_{f_1},\ldots,u_{f_n})$ where each of the coordinate functions satisfies 
 \[
 Lu_{f_j}=f_j \mbox{ in } M, \quad u_{f_j}=0 \mbox{ on } \p M, 
 \]
 with $f_j\in C_c^\infty(W)$, and where 
 \[
 f_j=0 \text{ on } \Omega, \quad j=1,\ldots,n.
 \]
 If the coefficients of $L$ are real analytic, then $(u_{f_1},\ldots,u_{f_n})$ is real analytic on $\Omega$.
\end{Lemma}
\begin{proof}
 Let $x_0\in \Mi$. If $x_0\in W$, we redefine $W$ as a smaller open set such that $x_0\notin W$. By Proposition \ref{prop_runge_consequence_interior_source_sec6}(b) we may find $f_1, \ldots, f_n \in C^{\infty}_c(W)$ such that the Jacobian matrix of $U = (u_{f_1},\ldots,u_{f_n})$ is the identity matrix at $x_0$. Thus $U$ is a coordinate system in some neighborhood $\Omega$ of $x_0$, and by shrinking $\Omega$ if necessary we may assume that $f_j = 0$ in $\Omega$.
 
 If the coefficients of $L$ are real analytic, the coordinate system $U$ is real analytic on $\Omega$ by elliptic regularity, since $f_j|_\Omega=0$.
 \end{proof}

The next lemma is an analogue of Lemma~\ref{J_is_local_dif}, where we considered the composition of Poisson embeddings.
If $W_1$ and $W_2$ are open subsets of $M_1$ and $M_2$ and $\phi: W_1\to W_2$ is a diffeomorphism, we define a mapping $\phi^*R_2:\Mi_2\to \mathcal{D}'(W_1)$ by
\begin{equation} \label{phir2_definition}
 [\phi^*R_2](x)f=R_2(x)(f\circ\phi^{-1}), 
\end{equation}
 with $x\in \Mi_2$ and $f\in C_c^\infty(W_1)$. From Proposition~\ref{R_basics} and Proposition~\ref{R_is_embedding}, it follows that $\phi^*R_2$ is a $C^k$ embedding $\Mi_2 \to H^{-s-1-k}(W)$, $k=0,1,\ldots$.

\begin{Lemma}\label{R_composition}
 Let $(M_1,g_1)$ and $(M_2,g_2)$ be compact manifolds with mutual boundary $\p M$. Let $W_j\subset M_j$, $j=1,2$, be open subsets and let $\phi$ be a diffeomorphism $W_1\to W_2$. 
  Assume that for some open set $B\subset \Mi_1$
  \[
 R_1(B) \subset (\phi^*R_2)(\Mi_2).
 \]
 Then
 \[
 J=(\phi^*R_2)^{-1}\circ R_1,
 \]
 is $C^\infty$ diffeomorphism $B\to J(B)\subset \Mi_2$.
\end{Lemma}
\begin{proof}
We prove the continuity of $J$ differently than we did in the corresponding situation in Lemma~\eqref{J_is_local_dif}. It follows from Proposition~\ref{R_basics} that we may continue $\phi^*R_2$ by zero to a continuous mapping $M_2\to H^{-s-1}(W_1)$. Let $E\subset \Mi_2$ be closed in $M_2$. It follows that $\phi^*R_2(E)$ is closed in $H^{-s-1}(W_1)$ by continuity and by compactness of $M_2$. Since $\phi^*R_2$ is injective on $\Mi_2$, we have that $I:=(\phi^*R_2)^{-1}$ is defined as a mapping $\phi^*R_2(\Mi_2)\to \Mi_2$ and we have that $I^{-1}(E)$ equals $\phi^*R_2(E)$, which is closed. Thus $I$ is continuous and consequently $J$ is continuous $B\to J(B)$ by Proposition~\ref{R_basics}.

Let $x_0\in \Mi_1$ and let $U=(u^2_{f_1},\ldots, u^2_{f_n})$ be solution coordinates of Lemma~\ref{solution_coords} on a neighborhood $\Omega_2$ of $J(x_0)$, where $f_k\in C_c^\infty(W_2)$, $k=1,\ldots,n$. In case $x_0\in W_1$, we may assume that $\phi(x_0)\notin \overline{\textrm{supp}(f_k)}$. We define $V:=J^*U$. Since $U$ is invertible, we have that $J=U^{-1}\circ V$. Since $J$ is continuous, the domain $\Omega_1:=J^{-1}(\Omega_2)$ of $V$ is an open neighborhood of $x_0$.  
We have $V=(v_{f_1}, \ldots, v_{f_n})$ where $v_{f_k}$ satisfies
\begin{align*}
 v_{f_k}(x)&=J^*u^2_{f_k}(x)=J^*(R_2(x)f_k)=R_2(J(x))((f_k\circ \phi)\circ \phi^{-1}) \\
 &=[\phi^*R_2](J(x))(f_k\circ\phi)=R_1(x)(f_k\circ\phi)=u_{f_k\circ \phi}^1(x).
\end{align*}
Since $\phi(x_0)\notin \overline{\textrm{supp}(f_k)}$, we have that each $u_{f_k\circ \phi}^1$ is $C^\infty$ near $x_0$. Thus $V$ is $C^\infty$ and consequently $J=U^{-1}\circ V$ is $C^\infty$. If $f\in \mathcal{D}'(W_1)$, we have that
\[
 u_f^1(x)=R_1(x)f=(\phi^*R_2\circ J)(x)f=[\phi^*R_2](J(x))f=u_{f\circ\phi^{-1}}^2(J(x)),
\]
and 
\[
du^1_f(x_0)\cdot X=d(u_{f\circ\phi^{-1}}^2\circ J)(x_0)\cdot X=du_{f\circ \phi^{-1}}^2|_{J(x_0)}\cdot (DJ(x_0)) X.
\]
It follows from Proposition~\ref{prop_runge_consequence_interior_source_sec6} that $DJ(x_0)$ is invertible. Since $J$ is also injective by Proposition~\ref{R_is_embedding}, it follows that $J:B\to J(B)$ is $C^\infty$ diffeomorphism.
\end{proof}

We prove next the main result of this subsection.

\begin{proof}[Proof of Theorem~\ref{global_determination_of_harmonic_functions_for_L}]
Assume that the source-to-solutions maps agree:
\begin{equation}\label{S_agree}
S^L_1\phi^*f=\phi^*S^L_2f, \quad f\in C_c^\infty(W_2).
\end{equation}
By shrinking $W_j$ if necessary, we may assume that $W_j$ are small geodesic balls.

Redefine $B$ to be the largest connected open set of $\Mi_1$ such that $R_1(B)\subset (\phi^*R_2)(\Mi_2)$. We have $\Wi_1 \subset B$. To see this, let $f\in C_c^\infty(W_1)$ and $x\in \Wi_1$. We have by using definitions and~\eqref{S_agree} that 
\begin{align}\label{too_many_eqs}
\nonumber R_1(x)f&=u^1_f(x)=(S^L_1f)(x)=[S^L_1(\phi^*(f\circ\phi^{-1}))](x)=[\phi^*S^L_2(f\circ\phi^{-1})](x)  \\ 
&=[S^L_2(f\circ\phi^{-1})](\phi(x))=u^2_{f\circ\phi^{-1}}(\phi(x))=R_2(\phi(x))(f\circ\phi^{-1}) \nonumber \\
&=[(\phi^*R_2)(\phi(x))]f.
\end{align}
Note that $\phi(x)\in \Mi_2$ since otherwise we have $u_f^1(x)=u_{f\circ \phi^{-1}}^2(\phi(x))=0$ for all $f\in C_c^\infty(W_1)$, which by Proposition \ref{prop_runge_consequence_interior_source_sec6} leads to a contradiction to the fact that $x_1\in \Mi_1$. We conclude that $B\neq \emptyset$. We show that $B$ is closed in $\Mi_1$ and thus $B=\Mi_1$.

On $B$ we define
\[
 J=(\phi^*R_2)^{-1}\circ R_1
\]
which is $C^\infty$ diffeomorphism $B\to J(B)$ by Lemma~\ref{R_composition}. We have for $x\in B$ and $f\in C_c^\infty(W_2)$ that $R_1(x)(f\circ\phi)=[\phi^*R_2](J(x))(f\circ\phi)$ and thus, by the definition \eqref{phir2_definition}, 
 \begin{equation}\label{sols_to_sols}
 u_{f\circ\phi}^1(x)=u_{f}^2(J(x)), \qquad x \in B.
 \end{equation}
 By~\eqref{too_many_eqs}, we directly also have that
 \begin{equation}\label{J_is_phi}
  J|_{\Wi_1}=\phi : \Wi_1 \to \Wi_2.
 \end{equation}

 To show that $B$ is closed, we argue by contradiction and let $p_k\to x_1\in \p B\setminus B$, with $p_k\in B$. By passing to a subsequence, we have $x_2:=\lim_kJ(p_k)\in M_2$. We have that $x_2\in \Mi_2$; otherwise $u_f(x_1)=0$ for all $f\in C_c^\infty(W_1)$ by~\eqref{sols_to_sols}. Let 
 \[
 U=(u_{f_1}^2,u_{f_2}^2,\ldots, u_{f_n}^2)
 \]
 be solution coordinates as in Lemma~\ref{solution_coords} on a neighborhood $\Omega_2\subset \Mi_2$ of the point $x_2$. Here $f_j\in C_c^\infty(W_2)$, $j=1,\ldots,n$, and the limit $x_2$ is found by first passing to a subsequence. 
 
 We set
 \[
 V=(u_{f_1\circ\phi}^1,u_{f_2\circ\phi}^1,\ldots, u_{f_n\circ\phi}^1).
 \]
 We have $J=U^{-1}\circ V$ on $B$ near $x_1$ by~\eqref{sols_to_sols}. We have that $V$ is a solution coordinate system on a neighborhood $\Omega_1\subset \Mi_1$ of $x_1$, and the ``equation of injectivity'' 
 \[
 du_f^1|_{x_1}\cdot X=du_{f\circ \phi^{-1}}^2|_{x_2}\cdot D(U^{-1}\circ V) X,\quad f\in C_c^\infty(W_1),
 \]
holds for $X \in T_{x_1} M_1$. These facts can be proved similarly following the argument in Theorem~\ref{global_determination_of_harmonic_functions}. We redefine the sets $\Omega_1$ and $\Omega_2$, if necessary, so that we have $V(\Omega_1)=U(\Omega_2)$.

 We next show that
 \begin{equation}\label{to_be_shown2}
 R_1(x)=(\phi^*R_2)(U^{-1}\circ V(x)), \quad x\in \Omega_1.
 \end{equation}
 This will imply
 \[
 R_1(B\cup \Omega_1)\subset (\phi^*R_2)(\Mi_2),
 \]
 which is a contradiction and proves the theorem.
 
 Let $f\in C_c^\infty(W_1)$. We claim that
  \[
   u_f^1\circ V^{-1}=u_{f\circ \phi^{-1}}^2\circ U^{-1} \mbox{ on } V(\Omega_1)\subset \R^n.
  \]
If $x_1\in \Wi_1$ there is nothing to prove since then $x_1\in \Wi_1\subset B$ and it is a contradiction to $x_1\in \p B\setminus B$. Thus we may assume $x_1\in \Mi_1\setminus\overline{\textrm{supp}(f)}$. If $x_2\in \Wi_2$, then by~\eqref{J_is_phi} and by the injectivity of $J$ and continuity of $J^{-1}$ we have that
 \[
  \Wi_1\ni \phi^{-1}(x_2)= J^{-1}(x_2)=J^{-1}\lim_k(J(p_k))=\lim_kJ^{-1}J(p_k)=x_1.
 \]
Thus we may also assume that $x_2\in \Mi_2\setminus \Wi_2\subset \Mi_2\setminus\overline{\textrm{supp}(f\circ\phi^{-1})}$. 
 It follows that there are small geodesic balls $\Omega'_1\subset \Omega_1$ and $\Omega'_2\subset \Omega_2$ of $x_1$ and $x_2$, respectively, such that $u_f^1$ and $u_{f\circ\phi^{-1}}^2$ are $C^\omega$ on $\Omega'_1$ and $\Omega'_2$, and that $V(\Omega'_1)\subset U(\Omega'_2)$. Consequently, the functions 
 \begin{equation}\label{functions_ra}
 u_f^1\circ V^{-1}, \ u_{f\circ\phi^{-1}}^2\circ U^{-1}
 \end{equation}
 are real analytic on $V(\Omega'_1)$.
 
 The functions in~\eqref{functions_ra} agree on $V(B\cap \Omega_1)$ by \eqref{sols_to_sols}. Consequently, they agree on $V(\Omega'_1)$ by real analyticity. The set $V(\Omega'_1) \cap V(\Omega_1  \setminus \overline{B})$ is open and non-empty. Note that the functions in~\eqref{functions_ra} are real analytic on $V(\Omega_1 \setminus \overline{B})$. This is because $\overline{\textrm{supp}(f)}\subset \Wi_1\subset B$ and because $y\in \overline{\textrm{supp}(f\circ \phi^{-1})}\cap \Omega_2$ implies that 
 \begin{align*}
U(y)&\in U(\Wi_2\cap \Omega_2)\subset U(J(B)\cap \Omega_2)=U(J(B\cap \Omega_1)) \\
&=U((U^{-1}\circ V)(B\cap \Omega_1))=V(B\cap \Omega_1).
 \end{align*}
By real analyticity the functions in~\eqref{functions_ra} agree on $V(\Omega_1 \setminus \overline{B})$ and consequently on $V(\Omega_1)$.                                                                                                                                                                                                                                                                                                                                                                                                                                                                                        We have~\eqref{to_be_shown2} since it is equivalent to 
  \[
u_f^1\circ V^{-1}=u_{f\circ \phi^{-1}}^2\circ U^{-1} \mbox{ on } V(\Omega_1)\subset \R^n.
 \]
Thus $B$ extends to a neighborhood of the point $x_1\in \p B\setminus B$, which gives a contradiction. Thus $B$ is closed. Since $\Mi_1$ is connected, we conclude that $B=\Mi_1$.
 
 By Lemma~\ref{R_composition}, $J$ is $C^\infty$ diffeomorphism $\Mi_1\to J(\Mi_1)$. Inverting the role of $M_1$ and $M_2$, we have $J(\Mi_1)=\Mi_2$. Since by Lemma~\ref{solution_coords} we may locally represent $J$ as $U^{-1}\circ V$ near any point, where $U$ and $V$ are $C^\omega$ solution coordinates, we have that $J$ and $J^{-1}$ are real analytic. If $u_f^2$ solves $L_2u_f^2=f$, where $f\in C_c^\infty(W_2)$ and $u_f^2|_{\p M}=0$, then~\eqref{sols_to_sols} shows that $u_{f\circ \phi}^1=J^*u_f^2$.
\end{proof}

\subsection{Local determination of the coefficients}
We determine the coefficients of a quasilinear elliptic operator on open sets where the source-to-solution mapping of the operator is known. Precisely, we prove the following:
\begin{Proposition}\label{local_det_of_coef_for_Q}
 Let $(M_1,g_1)$ and $(M_2,g_2)$ be compact connected Riemannian manifolds with mutual boundary and assume that $Q_j$, $j=1,2$, are quasilinear operators of the form~\eqref{quasilinear_op_form} having coefficients $\mathcal{A}_j$, $\mathcal{B}_j$ satisfying \eqref{elliptic_a}--\eqref{l_assumption}.
 
 Let $W_j$, $j=1,2$, be open subsets of $M_j$, and assume that there is a diffeomorphism $\phi:W_1\to W_2$ so that the source-to-solutions maps $S_j$ for $Q_j$ agree in the sense that
\[
\phi^*S_2f=S_1\phi^*f,
\]
for all $f\in C_c^\infty(W_2)$ with $\norm{f}_{C^{\alpha}(M_2)}$ sufficiently small.

Let $\widetilde{W}\subset \subset \Wi_1$ be an open set. Then there is $\delta>0$ such that for all $x\in \widetilde{W}$ and $(c,\sigma)\in T_x\widetilde{W}$, with $\abs{c} +\abs{\sigma}_{g_1}\leq \delta$, we have
\begin{equation}\label{det_of_A}
\mathcal{A}_1(x,c,\sigma)=\phi^*\mathcal{A}_2(x,c,\sigma)=:\mathcal{A}(x,c,\sigma)
\end{equation}
and 
\begin{equation}\label{det_of_B}
\mathcal{B}_1(x,c,\sigma)-\phi^*\mathcal{B}_2(x,c,\sigma)=\mathcal{A}^{ab}(x,c,\sigma)(\Gamma(g_1)_{ab}^k-\Gamma(\phi^*g_2)_{ab}^k)\sigma_k.
\end{equation}
 \end{Proposition}

To prove the proposition, we begin with the following observation. 
\begin{Lemma}\label{all_equations}
 Assume the conditions and notation in Proposition~\ref{local_det_of_coef_for_Q}. 
 Let $x\in \Wi_1$ and let $U_2$ be coordinates on a neighborhood $\Omega_2\subset \subset W_2$ of $\phi(x)\in \Wi_2$. Let $U_1=\phi^*U_2$ be coordinates on a neighborhood $\phi^{-1}(\Omega_2)\subset \subset W_1$ of the point $x$. 
 
 Denote $\Omega=U_2(\Omega_2)\subset \R^n$. There is $\delta' >0$ such that for all $v\in C_c^\infty(\Omega)$ with $\norm{v}_{C^{2,\alpha}(\Omega)}\leq\delta'$ we have
 \[
 \tilde{Q}_1 v=\tilde{Q}_2 v \mbox{ on } \Omega.
 \]
 Here $\tilde{Q}_j$ are the coordinate representations of $Q_j$ in coordinates $U_j$, $j=1,2$.
\end{Lemma}
\begin{proof}
 Let $x\in \Wi_1$ and let $U_1$ and $U_2$ be coordinate systems as in the statement of the lemma. Let $\eps_j,\delta_j>0$, $j=1,2$, be as in the definition \eqref{s_w_definitiona}--\eqref{s_w_definitionb} of the source-to-solution mapping of $Q_j$. Set $\eps=\min(\eps_1,\eps_2)$. 
 Let $\delta'>0$ be such that if $v\in C_c^\infty(\Omega)$ satisfy $\norm{v}_{C^{2,\alpha}(\Omega)}\leq\delta'$, then 
 \[
 \min\big(\norm{Q_2(v\circ U_2)}_{C^{\alpha}(W_2)},\norm{\phi^*(Q_2(v\circ U_2))}_{C^{\alpha}(W_1)}\big)\leq\eps
 \]
 and 
 \[
  \norm{v\circ U_2}_{C^{2,\alpha}(W_2)}\leq\delta_2.
 \]
 Such a $\delta'$ can be found since $\mathcal{B}_2(x,0,0)=0$ by assumption~\eqref{b_assumption}, since $Q_2:C^{2,\alpha}(W_2)\to C^{\alpha}(W_2)$ is continuous and since composing a compactly supported function with $C^\infty$ diffeomorphism is continuous on H\"older spaces.
 
 Let $v\in C_c^\infty(\Omega)$ with $\norm{v}_{C^{2,\alpha}(\Omega)}\leq\delta'$. Then the problem $Q_2u_F^2=F$ in $M_2$ with $u_F^2|_{\p M}=0$ where 
 \[
 F:=Q_2(v\circ U_2)\in C^\infty_c(W_2)
 \]
 has a unique solution $u_F^2$ with $\norm{u_F^2}_{C^{2,\alpha}(M_2)}\leq\delta_2$ by Proposition~\ref{prop_nonlinear_wellposedness}. Let us extend $v\circ U_2$ by zero onto $M_2$. By the condition $\mathcal{B}_2(x,0,0)=0$ and we have
 \[
  Q_2(v\circ U_2)=F \text{ on } M_2 \text{ and } v\circ U_2=0 \text{ on } \p M.
 \]
Since $\norm{v\circ U_2}_{C^{2,\alpha}(M_2)}\leq\delta_2$, by uniqueness we have that
 \begin{equation}\label{uisv}
 u_F^2=v\circ U_2.
 \end{equation}

 By the definition of $\delta'$ we have that $Q_1u_F^1=\phi^*F$ with $u_F^1|_{\p M}=0$ has a unique solution with $\norm{u_F^1}_{C^{2,\alpha}(M_1)}\leq\delta_1$. Since the source-to-solutions mappings agree, we have 
 \begin{equation*}
 u_F^1|_{W_1}=S_1\phi^*F=\phi^*S_2F=\phi^*(u^2_F|_{W_2}).
 \end{equation*}
By using $\phi|_{\Omega_1}=U_2^{-1}\circ U_1$ it follows that
 \begin{equation}\label{sols_agree}
  u_F^1\circ U_1^{-1}=u_F^2\circ U_2^{-1} \text{ on } \Omega.
 \end{equation}
We also have on $W_2$ that
 \begin{equation}\label{Q1toQ2}
 Q_2u_F^2=F=F\circ\phi\circ\phi^{-1}=(Q_1u_F^1)\circ\phi^{-1}.
 \end{equation}
 Using the coordinate invariance of $Q_j$, we have by \eqref{uisv}--\eqref{Q1toQ2} that
 \begin{align*}
  \tilde{Q}_2v&=\tilde{Q}_2(u_F^2\circ U_2 ^{-1})=(Q_2u_F^2)\circ U_2^{-1}=(Q_1u_F^1)\circ (U_2\circ \phi)^{-1} \\
  &=(Q_1u_F^1)\circ U_1^{-1}=\tilde{Q}_1(u_F^1\circ U_1^{-1})=\tilde{Q}_1(u_F^2\circ U_2^{-1})=\tilde{Q}_1v. \qedhere
 \end{align*}
\end{proof}
The lemma tells us that we can use any test function $v$ with small enough $C^{2,\alpha}$ norm to solve for the coefficients of $\tilde{Q}_j$ from the equation $\tilde{Q}_1v=\tilde{Q}_2v$ in the coordinates $U_1$ and $U_2$. The local determination result of Proposition~\ref{local_det_of_coef_for_Q} is a consequence of this observation. Its proof is similar to that of Proposition~\ref{morphism_to_homothety} where we used harmonic functions to solve for (a multiple of) the metric in the Calder\'on problem.

\begin{proof}[Proof of Proposition~\ref{local_det_of_coef_for_Q}]
Let $x_0\in \Wi_1$ and let $U_2$ be coordinates on a neighborhood $\Omega_2\subset \subset \Wi_2$ of $\phi(x_0)\in \Wi_2$ and let $U_1=\phi^*U_2$ be coordinates on a neighborhood $\Omega_1=\phi^{-1}(\Omega_2)$ of $x_0$. Denote $\Omega=U_2(\Omega_2)=U_1(\Omega_1)$. By the Lemma~\ref{all_equations} above, we have that there is $\delta'>0$ such that for all $v\in C_c^\infty(\Omega)$ with $\norm{v}_{C^{2,\alpha}(\Omega)}\leq \delta'$ we have
\begin{equation}\label{test_v}
\tilde{Q}_1v=\tilde{Q}_2v.
\end{equation}
Here $\tilde{Q}_j$ are the coordinate representations of $Q_j$ in coordinates $U_j$, $j=1,2$. 

We construct the test functions we use. We may assume that $U_1(x_0)=U_2(\phi(x_0))=0$. Let $r>0$ so that $B(0,r)\subset \subset \Omega\subset \R^n$, and let $\chi:\Omega\to \R$ be a cutoff function, which is $1$ on $\overline{B}(0,r)$ and vanishes outside $\overline{B}(0,r')\subset\subset \Omega$ for some $r'>r$. Define $\widehat{\Omega}=U_2^{-1}(B(0,r))$.
If $y\in \widehat{\Omega}$ and $(c,\sigma)\in \mathbb{R}\times T_{y}^*\widehat{\Omega}$, define a function $v=v_{(y,c,\sigma,A)}\in C_c^\infty(\Omega)$ in the coordinates $U_2$ as
\begin{equation}\label{testfunction}
 v(x)=\chi(x)(c+\sigma\cdot (x-y)+\frac{1}{2}A(x-y)\cdot (x-y)),
\end{equation}
where $A$ is a symmetric $n\times n$-matrix.
Then we have
\[
 v(y)=c, \  dv(y)=\sigma  \mbox{ and } \p_{ab}v(y)=A_{ab}.
\]
There is $\delta>0$ such that for all $y\in \widehat{\Omega}$ and $(c,\sigma)\in \mathbb{R}\times T_{y}^*\widehat{\Omega}$ with $\abs{c}+\abs{\sigma}_{g_2}+ \norm{A}\leq 2\delta$, we have
\begin{equation}\label{lessthandelta}
 \norm{v_{(y,c,\sigma,A)}}_{C^{2,\alpha}(\Omega)}<\delta'.
\end{equation}
Here $\norm{A}$ is the Hilbert-Schmidt norm of matrices and $\delta$ is independent of $y\in \widehat{\Omega}$.

Let $y\in \widehat{\Omega}$ and let $(c,\sigma)\in \mathbb{R}\times T_{y}^*\widehat{\Omega}$ with $\abs{c}+\abs{\sigma}_{g_2}\leq \delta$. Let $A$ be a matrix with $\norm{A}\leq \delta$ and let $v=v_{(y,c,\sigma,A)}$ be the function defined in~\eqref{testfunction}. Since $\norm{v}_{C^{2,\alpha}(\Omega)}< \delta'$ by~\eqref{lessthandelta}, the equation~\ref{test_v} implies that
 \begin{align}\label{equation_to_solve_from}
 \mathcal{A}_1&(x,v(x),dv(x))^{ab} (\p_{ab}v(x)-\Gamma_{1,ab}^k(x)\p_kv(x))+\mathcal{B}_1(x,v(x),dv(x)) \\
 &= \mathcal{A}_2(x,v(x),dv(x))^{ab} (\p_{ab}v(x)-\Gamma_{2,ab}^k(x)\p_kv(x))+\mathcal{B}_2(x,v(x),dv(x)), \nonumber
 \end{align}
 for $x\in \Omega$. Here $\Gamma_{j,ab}^k$ are the Christoffel symbols of $g_j$, and $\mathcal{A}_j$ and $\mathcal{B}_j$ are understood as their coordinate representations in coordinates $U_j$.

 Let us first choose as the matrix $A$ a one with $\norm{A}\leq \delta$ and that satisfies have
 \[
  A \perp \mathcal{A}_1(y,v(y),dv(y))-\mathcal{A}_2(y,v(y),dv(y)).
 \]
 Here $\perp$ is defined with respect to the Hilbert-Schmidt inner product of matrices. We have $\norm{v}_{C^{2,\alpha}(\Omega)}\leq 2\delta$ and thus the equation~\ref{equation_to_solve_from} holds. It follows that 
 \begin{equation}\label{cond_for_B}
  \mathcal{B}_1(y,c,\sigma)-\mathcal{B}_2(y,c,\sigma)=\big(\mathcal{A}_1^{ab}(y,c,\sigma)\Gamma_{1,ab}^k(y)-\mathcal{A}_2^{ab}(y,c,\sigma)\Gamma_{2,ab}^k(y)\big)\sigma_k.
 \end{equation}
 We next choose the matrix $A$ as 
 \[
  A_{ab}=\rho(\mathcal{A}^{ab}_1(y,c,\sigma)-\mathcal{A}^{ab}_2(y,c,\sigma)),
 \]
 where $\rho>0$ is a number small enough so that $\norm{A}\leq \delta$.
It follows from~\eqref{equation_to_solve_from} and by using~\eqref{cond_for_B} that
\[
 \rho\norm{\mathcal{A}_1(y,c,\sigma)-\mathcal{A}_2(y,c,\sigma)}^2=0,
\]
where $\norm{\cdot}$ is the Hilbert-Schmidt norm of matrices.
Thus
\begin{equation}\label{solveda}
 \mathcal{A}_1^{ab}(y,c,\sigma)=\mathcal{A}_2^{ab}(y,c,\sigma),
\end{equation}
which combined with~\eqref{cond_for_B}, yields
\begin{equation}\label{solved}
   \mathcal{B}_1(y,c,\sigma)-\mathcal{B}_2(y,c,\sigma)=\mathcal{A}_1^{ab}(y,c,\sigma)(\Gamma_{1,ab}^k(y)-\Gamma_{2,ab}^k(y))\sigma_k.
\end{equation}

The equations~\eqref{solveda} and~\eqref{solved} hold for all $(y,c,\sigma)\in \widehat{\Omega}\times \R \otimes T^*\widehat{\Omega}$ with $\abs{c}+\abs{\sigma}_{g_2}\leq \delta$. Recall that~\eqref{solveda} and~\eqref{solved} are equations for the coordinate representations of $\mathcal{A}_1, \mathcal{B}_1$ and $\mathcal{A}_2, \mathcal{B}_2$ in coordinates $U_1$ and $U_2$ respectively, where $U_1=\phi^*U_2$. Thus we have, by redefining $\delta$ if necessary, that
\begin{align*}
 \mathcal{A}_1&(x,c,\sigma)=\phi^*\mathcal{A}_2(x,c,\sigma)=:\mathcal{A}(x,c,\sigma) \\
  \mathcal{B}_1&(x,c,\sigma)-\phi^*\mathcal{B}_2(x,c,\sigma)=\mathcal{A}^{ab}(x,c,\sigma)(\Gamma(g_1)_{ab}^k-\Gamma(\phi^*g_2)_{ab}^k)\sigma_k,
\end{align*}
for $(x,c,\sigma)\in (\phi^{-1}\widehat{\Omega})\times \R \otimes T^*(\phi^{-1}\widehat{\Omega})$ with $\abs{c}+\abs{\sigma}_{g_1}\leq \delta$. Let $\widetilde{W}$ be an open set compactly contained in $\Wi_1$. Since we may carry out the argument above on a neighborhood of any point $x_0\in \Wi_1$, we have the claim by compactness for some $\delta>0$.
 \end{proof}
 
\subsection{Global determination of the coefficients}
We prove the main theorem of this section.
\begin{proof}[Proof of Theorem~\ref{qlin_main_thm}]
By Theorem~\ref{global_determination_of_harmonic_functions_for_L} we know that there is a real analytic diffeomorphism $J:\Mi_1\to \Mi_2$ that satisfies $J|_{\Wi_1}=\phi:\Wi_1\to \Wi_2$. It follows that $J^*\mathcal{A}_2$ and $J^*\mathcal{B}_2$, which are given by
\[
 (J^*\mathcal{A}_2)(x,c,\sigma)=((DJ)^{-1})^T|_{J(x)}\mathcal{A}_2(J(x),c,J^{-1*}\sigma)(DJ)^{-1}|_{J(x)}
\]
and
\begin{equation}\label{J_transformed}
 (J^*\mathcal{B}_2)(x,c,\sigma)=\mathcal{B}_2(J(x),c,J^{-1*}\sigma),
\end{equation}
are real analytic for all $x\in \Mi_1$ and $(c,\sigma)\in \R\times T^*_x\Mi_1$. By Proposition~\ref{local_det_of_coef_for_Q} we have that there is a non-empty open set $\widetilde{W}\subset \Wi_1$ and $\delta>0$ such that for all $(c,\sigma)\in \R\times T_x^*\widetilde{W}$ with $\abs{c}+\abs{\sigma}_{g_1}\leq \delta$, we have
\begin{align}\label{finalclaim}
 \mathcal{A}_1&(x,c,\sigma)=J^*\mathcal{A}_2(x,c,\sigma) \text{ and } \\
 \mathcal{B}_1&(x,c,\sigma)-J^*\mathcal{B}_2(x,c,\sigma)=\mathcal{A}^{ab}(x,c,\sigma)(\Gamma(g_1)_{ab}^k(x)-\Gamma(J^*g_2)_{ab}^k(x))\sigma_k, \nonumber
\end{align}
where $\mathcal{A}=\mathcal{A}_1=J^*\mathcal{A}_2$. By real analyticity, we have~\eqref{finalclaim} for all $x\in \Mi_1$ and $(c,\sigma)\in \R\times T_x^*\Mi_1$.
\end{proof}
 
\begin{proof}[Proof of Corollary~\ref{qlin_cor_resolve_thm}]
(1) Let $(x,c,\sigma)\in \Mi_1\times \mathbb{R}\otimes T^*\Mi_1$. Assume that $\mathcal{A}_1(x,c,\sigma)$ is $s$-homogeneous in the $\sigma$-variable and that $\mathcal{B}_1$ and $\mathcal{B}_2$ are $s'$-homogeneous with $s'\neq s+1$. It follows from~\eqref{J_transformed} that 
 $J^*\mathcal{B}_2$ 
 is $s'$-homogeneous in the $\sigma$-variable. By Theorem~\ref{qlin_main_thm} we know that $J^*\mathcal{A}_2=\mathcal{A}_1$ is $s+1$ homogeneous in the $\sigma$-variable. Thus  
\begin{equation*}\label{gammasigma}
\Gamma(x,c,\sigma):=\mathcal{A}^{ab}(x,c,\sigma)(\Gamma(g_1)_{ab}^k(x)-\Gamma(J^*g_2)_{ab}^k(x))\sigma_k, 
 \end{equation*}
 where $\mathcal{A}=\mathcal{A}_1=J^*\mathcal{A}_2$, is $(s+1)$-homogeneous in the $\sigma$-variable. Since 
 %
 \[
  \mathcal{B}(x,c,\sigma):=\mathcal{B}_1(x,c,\sigma)-J^*\mathcal{B}_2(x,c,\sigma)
 \]
 is $s'$-homogeneous, $s'\neq s+1$, in the $\sigma$-variable, we must have 
 \[
  \mathcal{B}_1-J^*\mathcal{B}_2=0 \text{ and } \Gamma(x,c,\sigma)=0
 \]
by Theorem~\ref{qlin_main_thm}.

 (2) Assume that $\phi$ is an isometry. Since $J|_{\Wi_1}=\phi|_{\Wi_1}$, we have $\Gamma(g_1)_{ab}^k=\Gamma(\phi^*g_2)_{ab}^k=\Gamma(J^*g_2)_{ab}^k$ on $\Wi_1$. Since $\Gamma(g_1)_{ab}^k-\Gamma(J^*g_2)_{ab}^k$ is a real analytic tensor field, which vanishes on $\Wi_1$, it vanishes on $\Mi_1$. Thus $\Gamma(x,c,\sigma)=0=\mathcal{B}(x,c,\sigma)$ for all $(x,c,\sigma)\in \Mi_1\times \mathbb{R}\otimes T^*\Mi_1$.
 \end{proof}

\appendix

\section{Runge approximation results}\label{runge_apprx_sec}

Here we prove the Runge approximation results that are used repeatedly in this paper. In this section, which is independent of the other sections, we will assume that $(M,g)$ is a compact connected oriented Riemannian manifold with boundary, and $\dim(M) \geq 2$. We will mostly use the following easy consequence of Runge approximation.

\begin{Proposition} \label{prop_runge_consequence}
Let $\Gamma$ be a nonempty open subset of $\partial M$, and denote by $u_f$ the solution of $\Delta_g u = 0$ in $M$ with $u|_{\partial M} = f$.
\begin{enumerate}
\item[(a)] 
If $x \in \Mi \cup \Gamma$, $y \in M$ and $x \neq y$, there is $f \in C^{\infty}_c(\Gamma)$ such that 
\[
u_f(x) \neq u_f(y).
\]
\item[(b)] 
If $x \in \Mi \cup \Gamma$ and $v \in T_x^* M$, there is $f \in C^{\infty}_c(\Gamma)$ such that 
\[
du_f(x) = v.
\]
\end{enumerate}
\end{Proposition}

The main Runge approximation result is the following.

\begin{Proposition} \label{prop_runge_approximation_cauchy}
Let $s \geq 1$, let $\Gamma$ be a nonempty open subset of $\partial M$, and let $U \subset \subset M^{\mathrm{int}}$ be a domain with $C^{\infty}$ boundary such that $M \setminus \overline{U}$ is connected. Let also $L$ be a second order uniformly elliptic differential operator on $M$, and let $\mathcal{P}$ be the Poisson operator for $L$. Then the set 
\[
\mathcal{R} = \{ \mathcal{P} f|_{U} \,;\, f \in C^{\infty}_c(\Gamma) \}
\]
is dense in the space $\mathcal{S} = \{ u \in H^s(U) \,:\, L u = 0 \text{ in $U$} \}$ with respect to the $H^s(U)$ norm.
\end{Proposition}

The proof is a standard Runge approximation argument, which boils down to the solvability of the adjoint equation in negative order Sobolev spaces and the unique continuation property. It will be convenient to embed $(M,g)$ in a closed manifold $(N,g)$ and to extend the operator $L$ in $M$ as a second order uniformly elliptic operator $A$ in $N$. In this way we avoid having to consider boundary values of solutions in negative order Sobolev spaces. The required solvability result will follow from the next lemma, where $A^*$ is the formal $L^2$-adjoint.

\begin{Lemma} \label{lemma_negative_sobolev_solvability}
Let $(N,g)$ be a closed manifold, and let $A$ be an elliptic second order differential operator on $N$ with kernel $\mathcal{N} = \{ v \in C^{\infty}(N) \,;\, Av = 0 \}$.

If $s \in \mathbb{R}$, then for any $F \in H^{-s}(N)$ satisfying $(F, v)_N = 0$ for all $v \in \mathcal{N}$ there is a unique solution $u \in H^{-s+2}(N)$ of 
\[
A^*u = F \text{ in $N$}.
\]
Moreover, if $F|_U = 0$ for some open set $U \subset N$, then for any $m \geq 0$ and any $V \subset \subset U$ one has 
\[
\norm{u}_{H^m(V)} \leq C_{m,V} \norm{F}_{H^{-s}(N)}.
\]
\end{Lemma}
\begin{proof}
Let $r \in \mathbb{R}$, and consider the map 
\[
T_r: H^{r+2}(N) \to H^r(N), \ \ T_r u = A^* u.
\]
By \cite[Theorem 19.2.1]{H3} this map is a Fredholm operator with finite dimensional kernel $\mathcal{N}^* := \{ v \in C^{\infty}(N) \,;\, A^* v = 0 \}$, and the range of $T_r$ is given by 
\[
\mathrm{Ran}(T_r) = \{ w \in H^r(N) \,;\, (w,v)_N = 0 \text{ for $v \in \mathcal{N}$} \}.
\]
In particular, the kernel and cokernel are independent of $r$. This proves the first statement.

Assume now that $F|_U \equiv 0$, let $V \subset \subset U$ and let $m \geq 0$. Let $\chi, \chi_1 \in C^{\infty}_c(U)$ satisfy $\chi = 1$ near $\overline{V}$ and $\chi_1 = 1$ near $\mathrm{supp}(\chi)$. Let also $Q \in \Psi^{-2}(N)$ be a parametrix for $A^*$, so that 
\[
QA^* = \mathrm{Id} + R
\]
where $R \in \Psi^{-\infty}(N)$. Then the solution $u = (A^*)^{-1} F$ of $A^*u = F$ in $N$ satisfies $u = QF - Ru = QF - R (A^*)^{-1} F$. Consequently 
\begin{align*}
\norm{u}_{H^m(V)} &\leq \norm{\chi u}_{H^m(N)} \leq \norm{\chi Q F}_{H^m(N)} + \norm{\chi R (A^*)^{-1} F}_{H^m(N)} \\
 &\leq \norm{\chi Q (1-\chi_1) F}_{H^m(N)} + \norm{\chi R (A^*)^{-1} F}_{H^m(N)}
\end{align*}
using that $F|_U \equiv 0$. Using the pseudolocal property of $Q$ and the fact that $R$ is smoothing, we get the estimate 
\[
\norm{u}_{H^m(V)} \leq C \norm{F}_{H^{-s}(N)}. \qedhere
\]
\end{proof}

We will also need the following simple lemma.

\begin{Lemma} \label{lemma_fix_orthogonality}
Let $(N,g)$ be a closed manifold, let $A$ be a second order uniformly elliptic differential operator on $N$, let $\mathcal{N}$ be the kernel of $A$, and let $\{ \psi_1, \ldots, \psi_m \}$ be an $L^2(N)$-orthonormal basis of $\mathcal{N}$. Let also $V \subset N$ be any nonempty open set. There exist $\eta_1, \ldots, \eta_m \in C^{\infty}_c(V)$ such that 
\[
(\eta_j, \psi_k)_N = \delta_{jk}, \qquad 1 \leq j, k \leq m.
\]
\end{Lemma}
\begin{proof}
Consider the map 
\[
T: L^2(V) \to \mR^m, \ \ T \eta = ( (\eta, \psi_1)_V, \ldots, (\eta, \psi_m)_V ).
\]
Then $T(a_1 \psi_1|_V + \ldots + a_m \psi_m|_V) = S \vec{a}$, where $S = ( (\psi_j, \psi_k)_V )_{j,k=1}^m$. Now if $S\vec{a} = 0$, then in particular $S \vec{a} \cdot \vec{a} = \norm{a_1 \psi_1 + \ldots + a_m \psi_m}_{L^2(V)}^2 = 0$ and thus $a_1 \psi_1|_V + \ldots + a_m \psi_m|_V = 0$. It follows that $a_1 \psi_1 + \ldots + a_m \psi_m = 0$ in $N$ by elliptic unique continuation, showing that $\vec{a} = 0$ and that $S$ is invertible.

Since the matrix $S$ is invertible, $T$ is surjective. Finally, since $C^{\infty}_c(V)$ is dense in $L^2(V)$ we have that $T(C^{\infty}_c(V))$ is a dense subspace of $\mR^m$. Since $\mR^m$ is finite dimensional, it follows that $T(C^{\infty}_c(V)) = \mR^m$. We can thus choose $\eta_j \in C^{\infty}_c(V)$ with $T \eta_j = e_j$ for $1 \leq j \leq m$.
\end{proof}

\begin{proof}[Proof of Proposition \ref{prop_runge_approximation_cauchy}]
Let $F \in (H^s(U))^*$ satisfy $F( \mathcal{P}f|_U) = 0$ for all $f \in C^{\infty}_c(\Gamma)$. By the Hahn-Banach theorem, it is enough to show that $F(u) = 0$ for all $u \in \mathcal{S}$.

Let $(M,g) \subset \subset (M_1,g) \subset \subset (N,g)$, where $M_1$ is a compact manifold with boundary and $N$ is a closed connected manifold. Let $A$ be an elliptic second order operator in $N$ such that $Aw|_M = L(w|_M)$ for $w \in C^{\infty}(N)$. Also, define $\hat{F} \in H^{-s}(N)$ by 
\[
\hat{F}(w) = F(w|_U), \qquad w \in H^s(N).
\]
It follows that $\mathrm{supp}(\hat{F}) \subset \overline{U}$. If $\hat{F}$ is not orthogonal to the kernel $\mathcal{N}$ of $A$, we need to modify it outside $M_1$ as follows. Let $V$ be an open subset of $N \setminus M_1$, let $\{ \psi_1, \ldots, \psi_m \}$ and $\{ \eta_1, \ldots, \eta_m \}$ be as in Lemma \ref{lemma_fix_orthogonality}, and define
\begin{equation} \label{eta_correction_formula}
\eta = - \sum_{j=1}^m \hat{F}(\psi_j) \eta_j.
\end{equation}
We define $\tilde{F} \in H^{-s}(N)$ with $\mathrm{supp}(\tilde{F}) \subset \overline{U} \cup \overline{V}$ by 
\[
\tilde{F} = \hat{F} + \eta.
\]
Lemma \ref{lemma_fix_orthogonality} ensures that $\tilde{F}$ is $L^2$-orthogonal to $\mathcal{N}$.

Now, by Lemma \ref{lemma_negative_sobolev_solvability} there is a unique solution $v \in H^{-s+2}(N)$ of 
\[
A^* v = \tilde{F} \text{ in $N$}.
\]
It will be convenient to choose smooth approximations $\tilde{F}_j = \hat{F}_j + \eta_j$, where $\hat{F}_j \in C^{\infty}_c(M^{\mathrm{int}})$ has support near $\overline{U}$ and $\hat{F}_j \to \hat{F}$ in $H^{-s}(N)$ and $\eta_j \in C^{\infty}_c(V)$ is defined as in \eqref{eta_correction_formula} but with $\hat{F}$ replaced by $\hat{F}_j$. Then $\tilde{F}_j \to \tilde{F}$ in $H^{-s}(N)$ as $j \to \infty$. We let $v_j \in C^{\infty}(N)$ solve 
\[
A^* v_j = \tilde{F}_j \text{ in $N$}.
\]

In particular, one has $A^* v = 0$ near $\partial M$, so $v$ is $C^{\infty}$ near $\partial M$ by elliptic regularity. We will next solve 
\begin{gather*}
L^* w_j = 0 \text{ in $M$}, \quad w_j|_{\partial M} = v_j|_{\partial M}, \\
L^* w = 0 \text{ in $M$}, \quad w|_{\partial M} = v|_{\partial M}.
\end{gather*}
Define $\varphi_j = v_j|_M - w_j$ and $\varphi = v|_M - w$. It follows that $\varphi_j \in C^{\infty}(M)$ and $\varphi \in H^{-s+2}(M)$ solve 
\begin{gather*}
L^* \varphi_j = \tilde{F}_j \text{ in $M^{\mathrm{int}}$}, \quad \varphi_j|_{\partial M} = 0, \\
L^* \varphi = \tilde{F} \text{ in $M^{\mathrm{int}}$}, \quad \varphi|_{\partial M} = 0.
\end{gather*}

Let $E$ be a bounded extension operator $H^s(M) \to H^s(N)$ that satisfies $\mathrm{supp}(Ew) \subset M_1^{\mathrm{int}}$ for all $w \in H^s(M)$. Since $F(\mathcal{P} f|_U) = 0$ for all $f \in C^{\infty}_c(\Gamma)$, we have 
\begin{align*}
0 &= F(\mathcal{P} f|_U) = \tilde{F}(E(\mathcal{P} f)) = \lim\,\tilde{F}_j(E(\mathcal{P} f))\\
 &= \lim\,(A^* v_j, E(\mathcal{P} f))_{N}.
\end{align*}
In the last expression, $E(\mathcal{P} f)$ is supported in $M_1^{\mathrm{int}}$ and $A^* v_j|_{M_1}$ vanishes outside a neighborhood of $\overline{U}$. Thus this expression can be understood as an integral over $M$. Integrating by parts, we get 
\begin{align*}
0 &= \lim\,(A^* v_j, \mathcal{P} f)_M = \lim\,(A^* \varphi_j, \mathcal{P} f)_M \\
 &= \lim\, (\partial_{\nu}^{A^*} \varphi_j, f)_{\partial M}.
\end{align*}
Here $\partial_{\nu}^{A^*}$ is the normal derivative associated with $A^*$, and we used that $\varphi_j|_{\partial M} = 0$. Now, if $W$ is a small neighborhood of $\partial M$, the higher regularity estimate in Lemma \ref{lemma_negative_sobolev_solvability} implies that $v_j \to v$ in $H^m(W)$ for any $m \geq 0$. Thus one has $\partial_{\nu}^{A^*} \varphi_j \to \partial_{\nu}^{A^*} \varphi$ in $H^m(\partial M)$ for any $m \geq 0$, and consequently 
\[
0 = (\partial_{\nu}^{A^*} \varphi, f)_{\partial M}.
\]
Since the previous result holds for all $f \in C^{\infty}_c(\Gamma)$, we see that $\varphi$ solves 
\[
L^* \varphi = 0 \text{ in } M^{\mathrm{int}} \setminus \overline{U}, \quad \varphi|_{\partial M} = 0, \quad \partial_{\nu}^{A^*} \varphi|_{\Gamma} = 0.
\]
Also, $\varphi$ is smooth in $M \setminus \overline{U}$. The unique continuation principle implies that $\varphi|_{M \setminus \overline{U}} = 0$. Thus $\varphi$ may be identified with an element of $H^{-s+2}_{\overline{U}}(N)$ (the distributions in $H^{-s+2}(N)$ supported in $\overline{U}$), and it follows that there exist $\psi_j \in C^{\infty}_c(U)$ with $\psi_j \to \varphi$ in $H^{-s+2}(N)$ (see e.g.\ \cite[Theorem 3.29]{McLean}).

Finally, let $u \in \mathcal{S}$, and let $E$ be a bounded extension operator from $H^s(U)$ to $H^s(N)$ so that $\mathrm{supp}(Ew) \subset M^{\mathrm{int}}$ for any $w \in H^s(U)$. Then, since  $\varphi \in H^{-s+2}_{\overline{U}}(N)$, 
\begin{align*}
F(u) &= \tilde{F}(Eu) = (A^* \varphi, Eu)_{N} \\
 &= \lim\, (A^* \psi_j, Eu)_{N} = \lim\, (A^* \psi_j, u)_{U}.
\end{align*}
Using that $\psi_j \in C^{\infty}_c(U)$ and $L u = 0$ in $U$, we may integrate by parts in $U$ and obtain that $F(u) = 0$ for all $u \in \mathcal{S}$. This concludes the proof.
\end{proof}

As a consequence of the Runge approximation property we can find global harmonic functions with prescribed $2$-jet at a given point, up to the natural restrictions given by the equation $\Delta_g u=0$. The existence of a local harmonic function with this property is found e.g.\ in \cite[Lemma A.1.1]{Wood_book}), but we give the details for completeness.

\begin{Proposition} \label{runge_prescribed_jet}
Let $\Gamma$ be a nonempty open subset of $\partial M$. Let $p \in \Mi \cup \Gamma$, let $a_0 \in \mR$, let $\xi_0 \in T_p^* M$, and let $H_0$ be a symmetric $2$-tensor at $p$ satisfying $\mathrm{Tr}_g(H_0) = 0$. There exists $f \in C^{\infty}_c(\Gamma)$ such that the solution of 
\[
-\Delta_g u = 0 \text{ in $M$}, \qquad u|_{\p M} = f
\]
satisfies $u(p) = a_0$, $du(p) = \xi_0$, and $\mathrm{Hess}_g(u)|_p = H_0$.
\end{Proposition}
\begin{proof}
It is enough to do the proof in the case where $p \in \Mi$. For if this has been done, and if $p \in \Gamma$, we may extend the manifold $M$ near $p$ to a slightly larger manifold $M_1$ so that $p \in \Mi_1$ and $\partial M \setminus \Gamma \subset \subset \partial M_1$. Since $p$ is an interior point of $M_1$, we can find a harmonic function $u_1$ in $M_1$ having the correct second order Taylor expansion at $p$ and satisfying $u_1|_{\partial M \setminus \Gamma} = 0$. Choosing $f = u_1|_{\partial M}$ implies that the solution of $-\Delta_g u = 0$ in $M$ with $u|_{\partial M} = f$ satisfies $u = u_1|_M$ and has the required behaviour at $p$.

Thus, assume that $p \in \Mi$. The proof will be given in four steps. \\

{\it Step 1.} First we find a local $g(p)$-harmonic function with prescribed $2$-jet at $p$. Let $x = (x^1, \ldots, x^n)$ be normal coordinates in a small geodesic ball $U$ centered at $p$. In these coordinates, let $\xi_0 = (\xi_0)_j \,dx^j|_p$ and $H_0 = H_{jk} \,dx^j \otimes dx^k|_p$. Define the function 
\[
u_0: U \to \mR, \ \ u_0(\exp_p(x^j \partial_j|_p)) = a_0 + (\xi_0)_j x^j + \frac{1}{2} H_{jk} x^j x^k.
\]
Clearly $u_0 \in C^{\infty}(\overline{U})$, $u_0(p) = a_0$ and $du_0(p) = \xi_0$. The Hessian, computed in normal coordinates, is given by 
\begin{align*}
\mathrm{Hess}_g(u_0)|_p &= (\partial_{x_j x_k} u_0 - \Gamma_{jk}^l \partial_{x_l} u_0) \,dx^j \otimes dx^k|_p = H_0
\end{align*}
since $\Gamma_{jk}^l|_p = 0$ in normal coordinates. Since $H_0$ is trace free it follows that $\Delta_{g(p)} u_0 = \sum_{j=1}^n H_{jj} = 0$ in $U$, i.e.\ $u_0$ is harmonic in $U$ with respect to the metric $g(p)$ with coefficients frozen at $p$. \\

{\it Step 2.} Next we find a local $g$-harmonic function near $p$ with $2$-jet $(0, 0, H_0)$. This is done by perturbing the functions $u_0$ from Step 1 in small balls, see e.g.\ \cite[Proposition 5.10.4]{Taylor}. Let $u \in C^{\infty}(\overline{B(p,\eps)})$ solve 
\[
\Delta_g u = 0 \text{ in $B(p,\eps)$}, \qquad u|_{\p B(p,\eps)} = u_0
\]
where $u_0(x) = \frac{1}{2} H_{jk} x^j x^k$ in normal coordinates. We may rescale $\tilde{x} = x/\eps$, $\tilde{u}(\tilde{x}) = \eps^{-2} u(\eps \tilde{x})$, so that $\tilde{u}$ solves (with derivatives taken with respect to $\tilde{x}$) 
\[
(g^{jk}(\eps \tilde{x}) \partial_{jk} + \eps \Gamma^l(\eps \tilde{x}) \partial_l) \tilde{u} = 0 \text{ in $B_1$}, \qquad \tilde{u}|_{\p B_1} = \tilde{u}_0|_{\p B_1} = \frac{1}{2} H_{jk} \tilde{x}^j \tilde{x}^k.
\]
Writing $\tilde{u} = \tilde{u}_0 + \tilde{w}$ where $\tilde{u}_0(\tilde{x}) = \eps^{-2} u_0(\eps \tilde{x}) = u_0(\tilde{x})$, and using that $g^{jk}(0) \partial_{jk} \tilde{u}_0 = 0$, we see that $\tilde{w}$ solves 
\[
(g^{jk}(\eps \tilde{x}) \partial_{jk} + \eps \Gamma^l(\eps \tilde{x}) \partial_l) \tilde{w} = \tilde{F} \text{ in $B_1$}, \qquad \tilde{w}|_{\p B_1} = 0
\]
where $\norm{\tilde{F}}_{H^{n/2+3}(B_1)} \leq C \eps$ since $g$ is smooth. By elliptic regularity (where the constants are uniform with respect to $0 < \eps < 1$) and Sobolev embedding, $\norm{\tilde{w}}_{C^2(\overline{B_1})} \leq C \eps$ where $C$ is uniform over small $\eps$. Then also $\norm{u - u_0}_{C^2(\overline{B(p,\eps)})} \leq C \eps$, which shows that there are local harmonic functions near $p$ with $2$-jet arbitrarily close to $(0, 0, H_0)$ at $p$.

We can make the $2$-jet at $p$ exactly equal to $(0, 0, H_0)$ as follows. Consider the operator 
\begin{gather*}
S: \{ \text{local harmonic functions near $p$} \} \to \mR \times T_p^* M \times (S^2_{\mathrm{tf}})_p M, \\
 Su = (u(p), du(p), \mathrm{Hess}_g(u)|_p)
\end{gather*}
where $(S^2_{\mathrm{tf}})_p M$ is the space of trace free symmetric $2$-tensors at $p$. The range of $S$ is a linear subspace of the finite-dimensional space $\mR \times T_p^* M \times (S^2_{\mathrm{tf}})_p M$, hence $\mathrm{Ran}(S)$ is closed. For any $H_0 \in (S^2_{\mathrm{tf}})_p M$ one has $(0, 0, H_0) \in \overline{\mathrm{Ran}(S)}$, and thus there is a local harmonic function near $p$ with $2$-jet $(0, 0, H_0)$. \\

{\it Step 3.} We will next find a local $g$-harmonic function near $p$ with $2$-jet $(a_0, \xi_0, H_0)$. In fact, the argument in Step 2 with the choice $u_0(x) = (\xi_0)_j x^j$ and scaling $\tilde{u}(\tilde{x}) = \eps^{-1} u(\eps \tilde{x})$ leads to local harmonic functions with $1$-jet at $p$ first arbitrarily close to $(0, \xi_0)$, and then exactly equal to $(0, \xi_0)$ as in the end of Step 2. Adding one of the functions obtained in Step 2 yields a local harmonic function with $2$-jet $(0, \xi_0, H_0)$, and adding a constant gives the $2$-jet $(a_0, \xi_0, H_0)$. \\

{\it Step 4.} Finally, to find a global harmonic function with prescribed $2$-jet at $p$, consider the operator 
\[
T: C^{\infty}_c(\Gamma) \to \mR \times T_p^* M \times (S^2_{\mathrm{tf}})_p M, \ \ f \mapsto (u_f(p), du_f(p), \mathrm{Hess}_g(u_f)|_p)
\]
where $u_f$ is the harmonic function in $M$ with $u|_{\p M} = f$. Given any $(a_0, \xi_0, H_0) \in \mR \times T_p^* M \times (S^2_{\mathrm{tf}})_p M$, Step 3 shows that there is a local harmonic function $u_0$ in a small geodesic ball $U = B_{\eps}(p)$ having $2$-jet $(a_0,\xi_0, H_0)$ at $p$. Now Proposition \ref{prop_runge_approximation_cauchy} implies that there is a sequence $(f_j) \subset C^{\infty}_c(\Gamma)$ such that $u_{f_j}|_U \to u_{0}$ in $H^{n/2+3}(U)$, hence $u_{f_j} \to u_0$ in $C^2(\overline{U})$ by Sobolev embedding. This shows that $T(C^{\infty}_c(\Gamma))$ is a dense subspace of $\mR \times T_p^* M \times (S^2_{\mathrm{tf}})_p M$, but since the last space is finite dimensional $T$ has to be surjective.
\end{proof}

We can now prove the consequence stated in the beginning of the section:

\begin{proof}[Proof of Proposition \ref{prop_runge_consequence}]
$\phantom{a}$ \\[5pt]
(a) First assume that both $x$ and $y$ are in $\Mi$. Let $U = B_1 \cup B_2$ where $B_1$ and $B_2$ are balls centered at $x$ and $y$, which are chosen in such a way that $U \subset \subset \Mi$ and $M \setminus \overline{U}$ is connected. Consider the harmonic function $u_0$ in $U$ with $u_0|_{B_1} = 1$ and $u_0|_{B_2} = 0$. By Proposition \ref{prop_runge_approximation_cauchy}, there exist $f_j \in C^{\infty}_c(\Gamma)$ such that 
\[
\norm{u_{f_j}|_U - u_0}_{H^{n/2+1}(U)} \to 0 \text{ as $j \to \infty$}.
\]
By Sobolev embedding we also have $\norm{u_{f_j}|_U - u_0}_{L^{\infty}(U)} \to 0$. Choosing $f = f_j$ for large enough $j$ implies that $u_f(x) \neq u_f(y)$.

Now assume that $x \in \Mi$ and $y \in \partial M$. We choose $\Gamma' \subset \subset \Gamma \setminus \{y\}$, and use the argument above to find $f \in C^{\infty}_c(\Gamma')$ with $u_f(x) \neq 0$. Then one also has $u_f(x) \neq u_f(y) = 0$. Next, the case where $x \in \Gamma$ and $y \in \Mi$ reduces to the previous case by interchanging $x$ and $y$. Finally, if $x \in \Gamma$ and $y \in \partial M$ with $x \neq y$, choose some $f \in C^{\infty}_c(\Gamma)$ with $f(x) \neq f(y)$ to obtain that $u_f(x) \neq u_f(y)$ as required. \\

\noindent (b) The result follows from Proposition \ref{runge_prescribed_jet}.
\end{proof}

Finally, let us give a Runge approximation result for a linear source problem used for studying the inverse problem for nonlinear equations.

\begin{Proposition} \label{prop_runge_approximation_cauchy_source}
Let $(M,g)$ be a compact manifold with boundary and let $s \geq 1$. Let $W$ be an open subset of $M$, and let $U \subset \subset M^{\mathrm{int}}$ be a domain with $C^{\infty}$ boundary such that $M \setminus \overline{U}$ is connected and $\overline{W} \cap \overline{U} = \emptyset$. Let also $L$ be a second order uniformly elliptic differential operator on $M$ which is injective on $C^{\infty}(M) \cap H^1_0(M)$, and let $K: C^{\infty}_c(W) \to C^{\infty}(M), \ f \mapsto u$ be the solution operator for the problem
\begin{align*}
Lu=f \text{ in } M, \qquad u|_{\p M} = 0.
\end{align*}
Then the set 
\[
\mathcal{R} = \{ Kf|_{U} \,;\, f \in C^{\infty}_c(W) \}
\]
is dense in the space $\mathcal{S} = \{ u \in H^s(U) \,:\, L u = 0 \text{ in $U$} \}$ with respect to the $H^s(U)$ norm.
\end{Proposition}

\begin{proof}
Let $F \in (H^s(U))^*$ satisfy $F(Kf|_U) = 0$ for all $f \in C^{\infty}_c(W)$. By the Hahn-Banach theorem, it is enough to show that $F(u) = 0$ for all $u \in \mathcal{S}$. As in the proof of Proposition~\ref{prop_runge_approximation_cauchy}, we take $(M,g) \subset \subset (M_1,g) \subset \subset (N,g)$ and extend $L$ to an elliptic second order operator $A$ on $N$. Given $F \in (H^s(U))^*$, we define $\hat{F} \in H^{-s}(N)$, $\mathrm{supp}(\hat{F}) \subset \overline{U}$, by 
\[
\hat{F}(w) = F(w|_U), \qquad w \in H^s(N). 
\]

Continuing as in the proof of Proposition~\ref{prop_runge_approximation_cauchy}, we may find $\tilde{F}\in H^{-s}(N)$ so that
\[
\tilde{F}=\hat{F}+\eta
\]
with $\textrm{supp}(\eta) \subset V$, where $V\subset N\setminus M_1$, and $\tilde{F}$ is $L^2$-orthogonal to the kernel of $A$. Moreover we may find $\tilde{F}_j\in C_c^\infty(M^{\textrm{int}}\cup V)$ with $\tilde{F}_j\to \tilde{F}$ in $H^{-s}(N)$, and using the same procedure as in the proof of Proposition \ref{prop_runge_approximation_cauchy} we can represent $\tilde{F}$ and $\tilde{F}_j$ as 
\[
A^* v = \tilde{F} \text{ in $N$}, \qquad A^* v_j = \tilde{F}_j \text{ in $N$}.
\]
Similarly, in $M$ we have 
\begin{gather*}
L^* \varphi_j = \tilde{F}_j \text{ in $M^{\mathrm{int}}$}, \quad \varphi_j|_{\partial M} = 0, \\
L^* \varphi = \tilde{F} \text{ in $M^{\mathrm{int}}$}, \quad \varphi|_{\partial M} = 0,
\end{gather*}
where $\varphi_j \in C^{\infty}(M)$, $\varphi \in H^{-s+2}(M)$, and $\varphi$ is smooth near $\p M$.

Let $E$ be a bounded extension operator $H^s(M) \to H^s(N)$ satisfying $\mathrm{supp}(Ew) \subset M_1^{\mathrm{int}}$ for all $w \in H^s(M)$. Since $F(Kf|_U) = 0$ for $f \in C^{\infty}_c(W)$, we have 
\begin{align*}
0 &= F(Kf|_U) = \tilde{F}(E(Kf)) = \lim\,\tilde{F}_j(E(Kf))\\
 &= \lim\,(A^* v_j, E(Kf))_{N}.
\end{align*}
In the last expression, $E(Kf)$ is supported in $M_1^{\mathrm{int}}$ and $A^* v_j|_{M_1}$ vanishes outside a neighborhood of $\overline{U}\subset M^{\mathrm{int}}$. Thus this expression can be understood as an integral over $M$. Integrating by parts, we get
\begin{align*}
0 &= \lim\,(L^* \varphi_j, Kf)_M = \lim (\varphi_j, LKf)_M =\lim (\varphi_j, f)_W.
\end{align*}
Here we used that $Kf$ and $\varphi_j$ vanish on $\p M$ and thus the boundary terms vanish. 
The higher regularity estimate in Lemma \ref{lemma_negative_sobolev_solvability} implies that 
$\varphi_j \to  \varphi$
in $H^m(W)$ for any $m \geq 0$, and consequently 
\[
0 = (\varphi, f)_W.
\]
Since the previous result holds for all $f \in C^{\infty}_c(W)$, we see that $\varphi$ solves 
\[
L^* \varphi = 0 \text{ in } M^{\mathrm{int}} \setminus \overline{U}, \quad \varphi|_W \equiv 0.
\]
Also, $\varphi$ is smooth in the connected set $M \setminus \overline{U}$. The unique continuation principle implies that $\varphi|_{M \setminus \overline{U}} = 0$. Thus $\varphi \in H^{-s+2}_{\overline{U}}(N)$. As in the proof of Proposition~\ref{prop_runge_approximation_cauchy}, we have $\psi_j\in C_c^\infty(U)$, $\psi_j\to \varphi$ in $H^{-s+2}(N)$. If $E$ denotes a bounded extension operator from $H^s(U)$ to $H^s(N)$ with $\mathrm{supp}(Ew) \subset M^{\mathrm{int}}$ for any $w \in H^s(U)$, we have
\begin{align*}
F(u) &= \tilde{F}(Eu) = (A^* v, Eu)_{N} = \lim\, (A^* \psi_j, Eu)_{N} = \lim\, (L^* \psi_j, u)_{U}.
\end{align*}
If $u\in \mathcal{S}$, we can integrate by parts to see that $F(u)=0$. 
\end{proof}

As a consequence, we obtain an analogue of Proposition \ref{prop_runge_consequence} where boundary sources are replaced by interior sources.

\begin{Proposition} \label{prop_runge_consequence_interior_source}
Let $(M,g)$ be a compact manifold with boundary, and let $L$ be a second order uniformly elliptic differential operator on $M$ which is injective on $C^{\infty}(M) \cap H^1_0(M)$. Let $W$ be a nonempty open subset of $M$, and denote by $u_f$ the solution of $Lu = f$ in $M$ with $u|_{\partial M} = 0$.
\begin{enumerate}
\item[(a)] 
If $x \in \Mi$, $y \in M$ and $x \neq y$, there is $f \in C^{\infty}_c(W)$ such that 
\[
u_f(x) \neq u_f(y).
\]
\item[(b)] 
If $x \in \Mi$ and $v \in T_x^* M$, there is $f \in C^{\infty}_c(W)$ such that 
\[
du_f(x) = v.
\]
\end{enumerate}
\end{Proposition}
\begin{proof}
The proof is analogous to that of Proposition \ref{prop_runge_consequence}, except that we need a version of Proposition \ref{runge_prescribed_jet} for $0$- and $1$-jets where $-\Delta_g$ is replaced by $L$ (this can be found in \cite[Theorem I.5.4.1]{BJS}), and we need to use the Runge approximation result in Proposition \ref{prop_runge_approximation_cauchy_source}.
\end{proof}

\section{Proofs of basic properties} \label{sec_appendix}

\subsection{On the embeddings $P$ and $R$}\label{proofs_of_smoothness}
We present here the proofs of the basic mapping properties of the Poisson embeddings $P$ and $R$.
\begin{proof}[Proof of Proposition~\ref{P_basics}]
 We first show that $P$ maps $M^\Gamma$ to $H^{-s}(\Gamma)$ when $s+1/2>n/2$. 
Let $x_0\in M^\Gamma$ and $f\in C^{\infty}_c(\Gamma)$. 
We have
 \[
 \abs{P(x)f}=\abs{u_f(x)}\leq \norm{u_f}_{L^\infty(M)}\leq C\norm{u_f}_{H^{s+1/2}(M)}\leq C'\norm{f}_{H^s(\partial M)}.
 \]
Here we used Sobolev embedding and a standard estimate for solutions of linear elliptic equations, see \cite[Proposition 5.11.2]{Taylor} (the result holds for noninteger $s$ by interpolation).  
 Since 
 \[
 H^{-s}(\Gamma) = (\widetilde{H}^s(\Gamma))^*
 \]
 where $\widetilde{H}^s(\Gamma)$ is the closure of $C^{\infty}_c(\Gamma)$ in $H^s(\partial M)$, see \cite[Theorem 3.3]{ChandlerWildeHewettMoiola} (the result is also true on manifolds), it follows that $P(M^\Gamma)\subset H^{-s}(\Gamma)$. 
 
 We show next that $P: M^\Gamma \to H^{-s-1}(\Gamma)$  is continuous. For this let $x_0\in M^\Gamma$, and let $\gamma$ be a smooth curve on $M$ with $\gamma(0)=x_0$. Since harmonic functions are $C^\infty$, we may use Taylor expansion to write
 \[
 (P(\gamma(t))-P(\gamma(0)))f=u_f(\gamma(t))-u_f(\gamma(0))=(du_f \cdot \dot{\gamma}(c_f(t)))t,
 \]
 where $c_f(t)\in [0,t]$ depends on $f\in \widetilde{H}^{s+1}(\Gamma)$. Then 
 \begin{align*}
  &\abs{(P(\gamma(t))-P(\gamma(0)))f}=\abs{(du_f \cdot \dot{\gamma}(c_f(t)))t} \\
  & \qquad \qquad \leq C \abs{du_f(\gamma(c_f(t)))} t \leq C \norm{du_f}_{L^{\infty}(M)} t \leq C \norm{du_f}_{H^{s+1/2}(M)} t  \\
  & \qquad \qquad \leq C \norm{u_f}_{H^{s+3/2}(M)} t \leq C \norm{f}_{H^{s+1}(\partial M)} t.
 \end{align*}
Thus $\norm{P(\gamma(t))-P(\gamma(0))}_{H^{-s-1}(\Gamma)} \leq C t$. In particular $P:M^\Gamma\to \mathcal{D}'(\Gamma)$ is continuous since the topology of $\mathcal{D}'(\Gamma)$ is weaker than that of $H^{-s-1}(\Gamma)$.

 To show that $P:M^\Gamma \to H^{-s-2}(\Gamma)$ is Fr\'echet differentiable, we proceed similarly by using Taylor expansion
 \[
 (P(\gamma(t))-P(\gamma(0)))f=u_f(\gamma(t))-u_f(\gamma(0))=du_f(\dot{\gamma}(0))t+\nabla^2u_f(c(t))t^2,
 \]
 where $\nabla^2 u_f(c(t))$ is a short hand notation for the corresponding quadratic form and $c(t)$ depends on $f\in H^s(\Gamma)$.
 Then
 \begin{align*}
 &\norm{P(\gamma(t))-P(\gamma(0)) -du_{(\cdot)}(\dot{\gamma}(0))t}_{H^{-s-2}(\Gamma)}=\sup_{\norm{f}_{\widetilde{H}^{s+2}(\Gamma)}=1}\abs{\nabla^2u_{f}(c(t))t^2} \\
 &\quad \leq C \sup_{\norm{f}_{\widetilde{H}^{s+2}(\Gamma)}=1}\norm{u_f}_{W^{2,\infty}(M)}t^2\leq C\sup_{\norm{f}_{\widetilde{H}^{s+2}(\Gamma)}=1}\norm{u_f}_{H^{s+2+1/2}(M)}t^2\\
 &\quad \leq C'\sup_{\norm{f}_{\widetilde{H}^{s+2}(\Gamma)}=1}\norm{f}_{H^{s+2}(\partial M)}t^2=C't^2.
 \end{align*}
 This shows that the Fr\'echet derivative $DP_{x_0}:T_{x_0}M\to H^{-s-2}(\Gamma)$ of $P$ is given by
 \[
 [DP_{x_0}(V)](f)=du_f(x_0)\cdot V.
 \]
 Consequently $P$ is Fr\'echet differentiable also as a mapping $P: M^\Gamma\to \mathcal{D'}(\Gamma)$ with derivative given above. The higher order Fr\'echet derivatives follow similarly by considering higher order terms in the Taylor expansion of $P(\gamma(t))$. 
\end{proof}

\begin{proof}[Proof of Proposition~\ref{R_basics}.]
The proof is completely analogous to the proof of Proposition~\ref{P_basics} above.
Let $x_0\in \Mi_1$ and $s+2>n/2$.
We have by elliptic regularity~\cite[Proposition 5.11.2]{Taylor} and by Sobolev embedding that 
 \[
 \abs{R(x)f}=\abs{u_f(x)}\leq \norm{u_f}_{L^\infty(M)}\leq C\norm{u_f}_{H^{s+2}(M)}\leq C'\norm{f}_{H^s(W)}.
 \]
 Thus $R$ maps $\Mi$ to $H^{-s}(W)$. 
 
 To show that $R: \Mi\to H^{-s-1}(W)$ is continuous let $f\in H^{s+1}(W)$ and let $\gamma$ be a smooth curve on $\Mi$ with $\gamma(0)=x_0$. We have 
 \[
 (R(\gamma(t))-R(\gamma(0)))f=u_f(\gamma(t))-u_f(\gamma(0))=(du_{f}\cdot\dot{\gamma}(c_f(t)))t,
 \]
 where $c_f(t)\in [0,t]$ depends on $f$. By elliptic regularity~\cite[Proposition 5.11.2]{Taylor} we have that
 \begin{align*}
  &\abs{(P(\gamma(t))-P(\gamma(0)))f}=\abs{du_{f}\cdot\dot{\gamma}(c_f(t))t} \\
  &\leq C\abs{du_{f}\cdot\dot{\gamma}(c_f(t))t}\leq C \norm{du_f}_{L^\infty(M)}t\leq C\norm{du_f}_{H^{s+2}(M)}t \\
  &\leq C\norm{u_f}_{H^{s+3}(M)}t\leq C\norm{f}_{H^{s+1}(W)}t.
 \end{align*}
 This shows that $R:\Mi\to H^{-s-1}(W)$ is continuous. By the same calculation, we see that $R$ can be extended continuously to a mapping $M\to \to H^{-s-1}(W)$ by defining $R|_{\p M}=0$.
 
 That $R$ is $k$ times Fr\'echet differentiable $R:\Mi \to H^{-s-1-k}(W)$, and $C^\infty$ smooth $R:\Mi\to \mathcal{D}'(W)$, follows as in the proof of Proposition~\ref{P_basics} above by using elliptic regularity~\cite[Proposition 5.11.2]{Taylor}.
\end{proof}

\subsection{Auxiliary results used in 2D Calder\'on problem}\label{2D_appx}
We prove the existence of the special isothermal coordinates used in Lemma~\ref{det_near_bndr_2d} to prove the determination result near the boundary for harmonic functions in dimension $2$.
\begin{proof}[Proof of Lemma~\ref{bndr_determination_sothermal_coordinates}]
  We first determine the metric on the boundary in boundary normal coordinates as usual:
 Let $p\in \Gamma$ and let $h\in C_c^\infty(\Gamma)$ be a function on $\Gamma\subset \p M$ with non-vanishing gradient near $p$ on $\Gamma$ and $h(p)=0$. Let $\psi_j$, $j=1,2$, be two sets of boundary normal coordinates defined on open neighborhoods $\Omega_j$ of $p$ on $M_j$ with $\psi_1|_{\Gamma\cap \Omega_1}=\psi_2|_{\Gamma\cap \Omega_2}=(h,0)$. By the result~\cite[p. 1106]{LU} in dimension $2$, the Dirichlet-to-Neumann map determines the metric on the boundary. Denoting the coordinates of $\psi_1(\Omega_1)\cap \psi_2(\Omega_2)\subset\R^2$ by $(x,t)$, with $x$ being the coordinate of $\Gamma$, this is
 \begin{equation}\label{2d_on_the_bndr}
 (\psi_1^{-1*}g_1)(x,0)=(\psi_2^{-1*}g_2)(x,0), \quad x\in \psi_1(\Gamma\cap \Omega_1)\cap \psi_2(\Gamma\cap \Omega_2). 
 \end{equation}
 We set $\Gamma'=\psi_1(\Gamma\cap \Omega_1)\cap \psi_2(\Gamma\cap \Omega_2)$.
 
 We next transform to isothermal coordinates. For this, let $u^j$, $j=1,2$, denote the globally harmonic functions on $(M_j,g_j)$ 
 solving
  \begin{align*}
  \Delta_{g_j}u^j&=0 \text{ on } M_j \\
  u^j&=h \text{ on } \p M 
  \end{align*}
  Since the gradient of $h$ is nonvanishing on $\Gamma$, so is the gradient of $u^j$ on $\Gamma$. Since $u^j$ are global harmonic functions, the fact that the DN maps agree on $\Gamma$ implies that
  \begin{equation}\label{C_data_for_u}
  \p_{\nu_1}u^1|_{\Gamma} = \p_{\nu_2}u^2|_{\Gamma}.
  \end{equation}
  
  Let $\tilde{u}^j = u^j\circ \psi_j^{-1}$, $j=1,2$, be the coordinate representation of $u^j$ in the boundary normal coordinates $\psi_j$. Since $\p_{\nu_j}=\p_t$ in coordinates $\psi_j$, we have by~\eqref{C_data_for_u} that
  $\p_t\tilde{u}^1=\p_t\tilde{u}^2$
  and thus 
  \begin{equation}\label{2d_gradients_on_bndr}
  d\tilde{u}^1=d\tilde{u}^2 \mbox{ on } \Gamma'.
  \end{equation}
  
  We form local harmonic conjugates to $\tilde{u}^j$ by solving
  \begin{equation} \label{isothermal_tilde_equation}
  \ast_j d\tilde{u}^j=d\tilde{v}^j, \quad j=1,2,
  \end{equation}
  for $\tilde{v}^j$ locally. Here $\ast_j$ is the Hodge star of (the coordinate representation of) the metric $g_j$. A local solution $\tilde{v}^j$ exists by Poincar\'e lemma since $\ast_j \tilde{u}^j$ is closed by the harmonicity of $\tilde{u}^j$ (the proof of the Poincar\'e lemma shows that $\tilde{v}^j$ is smooth up to the boundary). We may assume $\tilde{v}^j(0,0)=0$. Since now $\ast_j^2$ is $-\text{Id}$ on $1$-forms and $d^2=0$, we have that  $\tilde{v}^j$ is harmonic. Since near the origin $(0,0)$ we have $0\neq\abs{d\tilde{u}^j}^2_{g_j}dV_{g_j}=d\tilde{u}^j\wedge d\tilde{v}^j$, we have that $(\tilde{u}^j,\tilde{v}^j)$ is a coordinate system on a neighborhood $\Omega\subset \R^2 \cap \{ t \geq 0 \}$ of the origin $(0,0)=\psi_1(p)=\psi_2(p)$. Redefine $\Gamma'$ as $\Gamma'\cap \Omega$.
  
  Let us denote 
  \[
  T_j=(\tilde{u}^j,\tilde{v}^j):\Omega \to \R^2, \quad j=1,2.
  \]
  We have that the Jacobian matrices
  \[
  DT_j=[d\tilde{u}^j,d\tilde{v}^j]
  \]
  of $T_j$ agree on $\Gamma'$. This is due to~\eqref{2d_gradients_on_bndr} and because
  \begin{equation}\label{2d_gradients_on_bndr2}
  d\tilde{v}^1=\ast_1d\tilde{u}^1=\ast_2d\tilde{u}^2=d\tilde{v}^2
  \end{equation}
  holds since Hodge star operators $\ast_j$ agree in the coordinates $\psi_j$ by~\eqref{2d_on_the_bndr}.
  
  Finally we are ready to pass to isothermal coordinates
  \[
  U_j:=T_j\circ\psi_j.
  \]
  The coordinates $U_j = (u^j, v^j)$ are isothermal coordinates, since $\ast_{g_j} du^j = dv^j$ due to \eqref{isothermal_tilde_equation} and the coordinate invariance of $d$ and $\ast$.
  Even though we might have that the set $\Gamma'$ is not mapped to $\{ t=0 \}$ by $T_j$, we have
  \[
  T_1(x,0)=T_2(x,0), \quad x\in \Gamma'.
  \]
  This is because $\tilde{u}^1(x,0)=h(x)=\tilde{u}^2(x,0)$ and because
  \[
  \tilde{v}^1(x,0)=\int_0^x\frac{d}{ds}\tilde{v}^1(s,0) \,ds=\int_0^x\frac{d}{ds}\tilde{v}^2(s,0) \,ds=\tilde{v}^2(x,0).
  \]
  Here we used that the gradients $d\tilde{v}^j$ agree on $\Gamma'$ by~\eqref{2d_gradients_on_bndr2} and that we have $\tilde{v}^j(0,0)=0$. We set:
  \[
  \tilde{\Gamma}=T_1(\Gamma')=T_2(\Gamma')\subset \R^2.
  \]
  By defining $\Gamma_0:=\psi_1^{-1}(\Gamma')=\psi_2^{-1}(\Gamma')$, we have the property (1) of the claim.
  
  
Let $f\in C_c^\infty(\Gamma)$ and let $u_f^1$ and $u_f^2$ be the corresponding global harmonic functions on $(M_1,g_1)$ and $(M_2,g_2)$ respectively. Since the DN maps agree on $\Gamma$, one has $d\tilde{u}_1 = d\tilde{u}_2$ on the boundary near $0$ where $\tilde{u}_j = u_f^j \circ \psi_j^{-1}$ and $\psi_j$ are the boundary normal coordinates. Finally, we pass from boundary normal coordinates to isothermal coordinates by the maps $T_j$ constructed above. Since we have proved that the Jacobian matrices of $T_j$ agree near the origin on $\{ x_2 = 0 \}$, we have that the Cauchy data of $u_f^j$ agree in these isothermal coordinates. This is (2), which concludes the proof. 
\end{proof}

Finally, for completeness we give the standard characterization of isothermal coordinates in terms of the coordinate expression of the metric.

\begin{Lemma}\label{isotherm_char}
Let $(M,g)$ be a $2$-dimensional $C^\infty$ Riemannian manifold. Then a coordinate system $U=(u^1,u^2)$ satisfies 
\[
\ast du^1= du^2
 \]
if and only if $(U^{-1*} g)_{jk} = c \delta_{jk}$ for some smooth positive function $c$.
\end{Lemma}
\begin{proof}
 Assume first that $(u^1,u^2)$ satisfies $du^1=\ast du^2$. Then we have 
 \[
 du^1\wedge du^2=du^1 \wedge \ast du^1 =g(du^1,du^1) \,dV_g= \abs{g}^{1/2}g(du^1,du^1) \,du^1\wedge du^2.
 \]
 We set $c=\abs{g}^{1/2}$ and it follows that
 \[
 |du^1|_g^2=c^{-1}.
 \]
 Since $\ast^2=-\mathrm{Id}$ for $1$-forms on a $2$-dimensional manifold, we also have $du^1=-\ast du^2$. Using this, the argument above gives $|du^2|_g^2=c^{-1}$. We also have $0=du^2\wedge du^2= du^2 \wedge \ast du^1=g(du^2,du^1) \,dV_g$. Thus the coordinate system is isothermal:
 \begin{equation} \label{metric_form_isothermal}
 g^{11}=g^{22}=c^{-1}, \quad g^{12}=g^{21}=0.
 \end{equation}
 
 To prove the opposite, we assume that the metric has the form \eqref{metric_form_isothermal} in the $(u^1,u^2)$ coordinates. Then, for all $1$-forms $\alpha = \alpha_1 \,du^1 + \alpha_2 \,du^2$, we have 
 \begin{equation}\label{c-inv_for_1-forms}
 \alpha\wedge *du^2=g(\alpha,du^2)dV_g=c^{-1}\abs{g}^{1/2} \alpha_2 \, du^1\wedge du^2= du^1 \wedge \alpha.
 \end{equation}
 Since this holds for all $1$-forms $\alpha$, we have that $du^1 = -\ast du^2$. 
\end{proof}

\subsection{Wellposedness and linearization of quasilinear equations}
 
This subsection contains the proofs of Propositions \ref{prop_nonlinear_wellposedness} and \ref{prop_nonlinear_linearization}.

\begin{proof}[Proof of Proposition \ref{prop_nonlinear_wellposedness}]
Let $X = C^{\alpha}(M)$, $Y = C^{2,\alpha}(M) \cap H^1_0(M)$ and $Z = C^{\alpha}(M)$ be Banach spaces, where the norm and the topology of $Y$ is given by the $C^{2,\alpha}(M)$ norm. Define the map 
\[
F: X \times Y \to Z, \ \ F(f, u) = Q(u) - f.
\]
Then $F(0,0) = 0$ by \eqref{b_assumption}. To check that $F$ is continuously differentiable, we use that the coefficients of $Q$ are $C^{2,\alpha}$ and 
note that 
\begin{align*}
 &F(f+h, u+v)-F(f,u) = Q(u+v) -Q(u) - h \\
 &= \int_0^1\int_0^t \frac{d^2}{ds^2}Q(u+sv) \,ds \,dt +t\frac{d}{dt}\Big|_{t=0}Q(u+tv)- h.
\end{align*}
The expression for $\frac{d^2}{ds^2} Q(u+sv)$ in terms of the coefficients $\mA$ and $\mB$ and derivatives of $v$ is long. It suffices to say that it is an expression that contains up to second order derivatives of $\mA^{ab}$ and $\mB$ evaluated at $(x,u+sv,du+sdv)$, and up second order derivatives of $u$ and $v$. Thus there is a constant $C$, uniform over $(u, v, s)$ in a bounded set of $Y \times Y \times \mR$, such that
\[
 \Norm{\frac{d^2}{ds^2}Q(u+sv)}_{C^\alpha(M)}\leq C\norm{v}_{C^{2,\alpha}(M)}.
\]
We have
\[
 \frac{d}{dt}\Big|_{t=0}Q(u+tv)=L_u,
\]
where $L_u: Y \to Z$ is the linearization of $Q$ at $u \in Y$, given by 
\begin{align*}
L_u v &= \mA^{ab}(x,u,du) \nabla_a \nabla_b v + \frac{\partial \mA^{ab}}{\partial u}(x,u,du) (\nabla_a \nabla_b u)v \\
 &\qquad + \frac{\partial \mA^{ab}}{\partial \sigma_j}(x,u,du) (\nabla_a \nabla_b u)\partial_j v + \frac{\partial \mB}{\partial u}(x,u,du) v + \frac{\partial \mB}{\partial \sigma_j}(x,u,du) \partial_j v.
\end{align*}
The linearization $L_u$ depends continuously on $u \in C^{2,\alpha}(M)$. It follows that $F$ is continuously Fr\'echet differentiable $X\times Y\to Z$ (in the sense of~\cite[Definition 10.2]{RR}).

By the assumption \eqref{l_assumption}, $L = L_0: Y\to Z$ is invertible. Thus, the implicit function theorem in Banach spaces \cite[Theorem 10.6]{RR} yields that there is $\eps>0$ and an open ball $B = B_X(0,\eps)\subset X$ and a continuously differentiable mapping $T: B \to Y$ such that whenever $\norm{f}_{C^\alpha(M)}\leq \eps$ we have
\[
 F(f,T(f))=0.
\]
Moreover, there is $\delta>0$ such that (by redefining $\eps$ if necessary) $T(f)$ is the only solution to $F(f,u)=0$ whenever $\norm{f}_{C^\alpha(M)}\leq \eps$ and $\norm{u}_{C^{2,\alpha}(M)}\leq \delta$. Moreover for such $u$ and $f$ we have
\[
 \norm{u}_{C^{2,\alpha}(M)} \leq C \norm{f}_{C^{\alpha}(M)},
\]
since $T$ is continuously differentiable.
\end{proof}

\begin{proof}[Proof of Proposition \ref{prop_nonlinear_linearization}]
Let $f \in C^{\infty}_c(W)$. For $t > 0$ small, let $u_t \in C^{\infty}(M)$ be the unique small solution given by Proposition \ref{prop_nonlinear_wellposedness} of 
\[
Q(u_t) = tf, \qquad u_t|_{\partial M} = 0,
\]
which satisfies the estimate 
\begin{equation} \label{ut_solution_estimate}
\norm{u_t}_{C^{2,\alpha}(M)} \leq Ct \norm{f}_{C^{\alpha}(M)}.
\end{equation}
We write $u_t = t v_t$, so that for $t > 0$ one has 
\[
\mA^{ab}(x,tv_t, td v_t) \nabla_a \nabla_b v_t + \frac{1}{t} \mB(x, tv_t, t d v_t) = f, \qquad v_t|_{\partial M} = 0.
\]
Let also $v_0$ be the unique solution of 
\[
Lv_0 = f \text{ in $M$}, \qquad v_0|_{\partial M} = 0.
\]
We wish to show that 
\[
v_t \to v_0 \text{ in $C^1(M)$ as $t \to 0$.}
\]
If this is true, then also 
\[
 \frac{S(tf) - S(0)}{t} = \frac{t v_t|_W}{t} \to v_0|_W \text{ in $C^1(\overline{W})$ as $t \to 0$},
\]
which proves the theorem.

We have 
\begin{align*}
0 &= \mA^{ab}(x,tv_t, td v_t) \nabla_a \nabla_b v_t + \frac{1}{t} \mB(x, tv_t, t d v_t) - L v_0 \\ 
 &= L(v_t - v_0) + F_t
\end{align*}
where 
\begin{align*}
F_t &= \mA^{ab}(x,tv_t, td v_t) \nabla_a \nabla_b v_t + \frac{1}{t} \mB(x, tv_t, t d v_t) - L v_t \\
 &=(\mA^{ab}(x,tv_t, td v_t) - \mA^{ab}(x,0,0)) \nabla_a \nabla_b v_t \\
 &\quad + \frac{1}{t} \mB(x, tv_t, t d v_t) - \frac{\partial \mB}{\partial u}(x,0,0) v_t - \frac{\partial \mB}{\partial \sigma_j}(x,0,0) \partial_j v_t \\
 &= \left( \int_0^t \left[ \frac{\partial \mA^{ab}}{\partial u}(x, sv_t, s d v_t) v_t + \frac{\partial \mA^{ab}}{\partial \sigma_j}(x, sv_t, s d v_t) \partial_j v_t \right] ds \right) \nabla_a \nabla_b v_t \\
 &\quad +\frac{1}{t}\left( \int_0^tds\int_0^s \frac{d^2}{dr^2} \mB(x,rv_t,rd v_t)dr\right) \\
 &= \left( \int_0^t \left[ \frac{\partial \mA^{ab}}{\partial u}(x, sv_t, s d v_t) v_t + \frac{\partial \mA^{ab}}{\partial \sigma_j}(x, sv_t, s d v_t) \partial_j v_t \right] ds \right) \nabla_a \nabla_b v_t  \\
 &\quad +\frac{1}{t} \Big( \int_0^tds\int_0^s \Big(\frac{\p^2}{\p u^2}\mB(x,rv_t,rd v_t)v_t^2 \\
 & + 2 \frac{\p^2}{\p u\p \sigma_j}\mB(x,rv_t,rd v_t)v_t\p_jv_t + \frac{\p^2}{\p\sigma_k\p\sigma_j}\mB(x,rv_t,rd v_t)\p_kv_t\p_jv_t\Big)dr\Big).
\end{align*}
In the second to last equality we used that $\mB(x,0,0)=0$. 

We wish to show that 
\[
\norm{F_t}_{L^{\infty}(M)} \leq Ct.
\]
This is obtained from the expression for $F_t$ given above, the estimate 
\[
\norm{v_t}_{C^{2,\alpha}(M)} \leq C \norm{f}_{C^{\alpha}(M)}
\]
which follows from \eqref{ut_solution_estimate}, and the fact that the integrals in the expression for $F_t$ are from $0$ to $t$ etc. Since one also has 
\[
L(v_t-v_0) = -F_t, \qquad v_t-v_0|_{\partial M} = 0,
\]
estimates for the linear equation imply that 
\[
\norm{v_t-v_0}_{C^1(M)} \leq Ct.
\]
This concludes the proof.
\end{proof}

\bibliographystyle{alpha}

\end{document}